\documentclass[11pt,a4paper]{article}

\usepackage{epsf,epsfig,amsfonts,amsgen,amsmath,amstext,amsbsy,amsopn,amsthm}
\usepackage{amsmath,times,mathptmx}
\usepackage{amsfonts,amsthm,amssymb}
\usepackage{graphicx}
\usepackage{latexsym,bm}
\usepackage{indentfirst}
\usepackage{color}
\usepackage[colorlinks=true,anchorcolor=blue,filecolor=blue,linkcolor=blue,urlcolor=blue,citecolor=blue]{hyperref}
\usepackage{float}
\usepackage{tikz}
\usepackage{verbatim}
\usepackage{mathrsfs}

\evensidemargin=\oddsidemargin
\evensidemargin=\oddsidemargin\topmargin=-1.5cm

\newtheorem{theorem}{Theorem}[section]
\newtheorem{prob}{Problem}

\newtheorem{lemma}{Lemma}[section]

\newtheorem{proposition}{Proposition}[section]

\newtheorem{conj}{Conjecture}[section]
\newtheorem{claim}{Claim}[section]

\addtocounter{section}{0}
\newcommand{\eps}{\varepsilon}

\usepackage{indentfirst}

\begin{document}

\title{\bf \LARGE The Brualdi--Hoffman--Tur\'{a}n problem for odd wheels: 
Resolving a conjecture of Yu, Li and Peng\footnote{Supported by the National
Natural Science Foundation of China (Nos.\,12501471,\,12571369) and the Stability Support Program for Basic Research (No. WZJC202601003).}}

\author{{\bf Longfei Fang$^{a}$},~
{\bf Mingqing Zhai$^b$},~~{\bf Yuhan Zhang$^{b}$}\thanks{Corresponding author: yuhzhang@njust.edu.cn (Y. Zhang).} \\[2mm]
\small $^{a}$ School of Mathematics and Finance, Chuzhou University, Chuzhou, Anhui 239012, China\\
\small $^{b}$ School of Mathematics and Statistics, Nanjing University of Science and
Technology, \\
\small Nanjing, Jiangsu 210094, China
}

\date{}
\maketitle

\begin{abstract}
This paper addresses the Brualdi--Hoffman--Tur\'{a}n problem, 
which asks for the maximum adjacency spectral radius over graphs with a fixed number of edges 
that forbid a prescribed subgraph. We settle an open conjecture due to Yu, Li and Peng for odd wheel graphs \(W_{2k+1}\). 
A striking dichotomy appears in the extremal configurations: 
the forbidden‑\(W_5\) case and the general case \(k\ge 3\) rely on identical stability and dense-core reductions 
but produce completely different extremal graphs. 
For sufficiently large size $m$, we derive sharp spectral inequalities and give complete characterizations for graphs achieving equality. 
Our approach combines the edge-spectral stability theorem with the dense-core method and residual-function analysis.
Beyond solving the target conjecture, 
the developed framework provides a systematic approach for edge-spectral Tur\'{a}n problems
and implies several known results for odd cycles, friendship graphs and odd fans. 
Directions for future research on edge-spectral supersaturation are also discussed.
\end{abstract}

\begin{flushleft}

\end{flushleft}

\section{Introduction}

All graphs are finite, simple and undirected.  We write $e(G)$ for the
number of edges of $G$, $A(G)$ for its adjacency matrix and $\rho(G)$
for the spectral radius of $A(G)$.
The subgraph of $G$ induced by the vertex subset $S\subseteq V(G)$ is denoted by $G[S]$.
Isolated vertices do not affect
either $e(G)$ or $\rho(G)$.
Therefore, unless stated otherwise, all graphs considered throughout this paper are assumed to have all isolated vertices removed.

A \emph{$k$-wheel} $W_k$ is the graph obtained from joining every vertex of a $(k-1)$-cycle with
an additional central vertex.
The classical Tur\'{a}n problem for wheels seeks the maximum number of edges in an $n$-vertex \(W_{2k+1}\)-free graph.
For \(k=2\), this problem was solved by Dzido and Jastrzebski \cite{Dzido},
while Yuan \cite{Yuan} settled the case for all \(k\ge 3\) and sufficiently large $n$.
The corresponding spectral Tur\'an problem was essentially resolved by Cioab\u{a}, Desai, and Tait \cite{Cioaba},
and later completely settled by Zhang \cite{Zhangwenqian}.

On the other hand,
the problem of characterizing graphs of size $m$ with maximum spectral radius was first raised by Brualdi and Hoffman \cite{Brualdi}.
Research on this problem subject to a Tur\'{a}n-type constraint can be traced back to 1970,
when Nosal \cite{Nosal} proved that \(\rho(G)\leq\sqrt{m}\) for every triangle-free graph $G$ of size $m.$
Nikiforov \cite{Nikiforov} subsequently established that \(\rho(G)\leq\sqrt{(1-1/r)2m}\) holds for all \(K_{r+1}\)-free graphs $G.$
Related edge-spectral extremal problems have been
studied for cycles and complete bipartite graphs
\cite{Lixin,Zhailinshu}, fans, friendship graphs,
and theta graphs \cite{Lishuchao,Liu},
star forests \cite{Zhangyanting}, and disjoint
triangles \cite{Joyentanuj}.
See also \cite{FangLinZhai0} for related results
on degenerate Tur\'{a}n problems.
More recently, Li, Liu, and Zhang \cite{LiLiuZhangStability}
established an edge-spectral Erd\H{o}s--Stone--Simonovits theorem.
Specifically, for every graph $F$ with $\chi(F)=r+1\ge3$, every
$m$-edge $F$-free graph $G$ satisfies
$\rho^2(G)\le\left(1-1/r+o(1)\right)2m$
as $m\to\infty$. They also proved a corresponding edge-spectral
stability theorem for nearly extremal graphs.
In subsequent work, Li, Liu, and Zhang
\cite{LiLiuZhangStability1} proved that for every
color-critical graph $F$ with $\chi(F)=r+1\ge4$,
every sufficiently large $m$-edge $F$-free graph
$G$ satisfies
$\rho^2(G)\le (1-\frac1r)2m$,
with equality if and only if $G$ is a regular
complete $r$-partite graph.
Since $\chi(W_{2k+1})=3$, this sharp result does not apply to
odd wheels $W_{2k+1}$,

For every fixed integer $k\ge1$ and every sufficiently large integer $m$, we define
\begin{equation}\label{eq:gdef}
g_k(m):=\frac{k-1+\sqrt{4m-k^2+1}}2.
\end{equation}

A \emph{$k$-fan} is defined as $V_k=K_1\vee P_{k-1}$.
Yu, Li, and Peng \cite{YuLiPeng} proposed three conjectures concerning the Brualdi–Hoffman–Tur\'{a}n problem
for fan graphs and wheel graphs separately.

\begin{conj}\label{conj1}
For fixed $k\geq2$ and sufficiently large $m$, if $G$ is a
$V_{2k+1}$-free or
$V_{2k+2}$-free graph with $m$ edges, then
$\rho(G)\le g_k(m).$
Equality holds if and only if $G\cong K_k\vee (\frac mk-\frac{k-1}2)K_1$.
\end{conj}

\begin{conj}\label{conj2}
For fixed $k\geq2$ and sufficiently large $m$, if $G$ is a
$W_{2k+1}$-free graph with $m$ edges, then
$\rho(G)\le g_k(m).$
Equality holds if and only if $G\cong K_k\vee (\frac mk-\frac{k-1}2)K_1$.
\end{conj}

\begin{conj}\label{conj3}
For fixed $k\geq2$ and sufficiently large $m$, if $G$ is a
$W_{2k+2}$-free graph with $m$ edges, then
$\rho(G)\le \sqrt{4m/3}.$
Equality holds if and only if $G$ is a regular complete 3-partite graph.
\end{conj}

Conjecture \ref{conj1} was confirmed by Li, Zhao and Zou \cite{Lishuchao}.
As a consequence, they also resolved the corresponding problem for friendship graphs.
Conjecture~\ref{conj3} was confirmed by
Li, Liu, and Zhang~\cite{LiLiuZhangStability1},
since $W_{2k+2}$ is a color-critical graph
with $\chi(W_{2k+2})=4$.
Prior to the present work,
Conjecture~\ref{conj2} remained open.
In an earlier study, Li, Liu, and Zhai \cite{LiLiuZhaiWheel} characterized the spectral extremal graphs of fixed size that are completely wheel-free,
with distinct extremal constructions depending on the parity of $m$.
Motivated by this result, Yu, Zhang and Zhang \cite{YuZhangZhang2026}
further proposed a conjecture for
$\{W_t:t\ge 2k+1\}$-free graphs, asserting the same spectral bounds and the corresponding extremal structures as Conjecture \ref{conj2}.
By contrast, forbidding only the single odd wheel \(W_{2k+1}\) imposes a substantially weaker constraint.
The main challenge lies in the fact that its decomposition graph family contains even cycles,
which leads to a nonlinear number of edges within the neighborhood of each vertex.

For every integer $m\geq1$, we define
\begin{equation}\label{eq:gasymp}
 g^*(m):=\frac{1+\sqrt{4m+1}}2.
\end{equation}
Then $g^*(m)^2-g^*(m)=m$.
The main results of this paper are presented as follows.

\begin{theorem}\label{thm1}
There exists $m_0$ such that every $W_5$-free graph $G$ with
$m\geq m_0$ edges satisfies
$\rho(G)\le g^*(m)$.
Equality holds if and only if
$ G\cong K_{n,n}+M_S+M_T,$
where $n$ is an even number such that $n^2+n=m$, and
$M_S$ and $M_T$ are perfect matchings embedded in the two
parts of $K_{n,n}$.
\end{theorem}

\begin{theorem}\label{thm2}
Let $k\ge3$ be fixed.
There exists $m_0=m_0(k)$ such that every
$W_{2k+1}$-free graph $G$ with $m\ge m_0$ edges satisfies
$\rho(G)\le g_k(m)$.  Equality holds if and only if
$G\cong K_k\vee qK_1$,
where $q$ is a positive integer such that ${k \choose 2}+kq=m$.
\end{theorem}

By Theorems \ref{thm1} and \ref{thm2}, Conjecture \ref{conj2} is completely resolved.
The case \(k=2\) is fundamentally different from the general case \(k\ge 3\).
Our proof employs the edge-spectral stability theorem due to Li-Liu-Zhang \cite{LiLiuZhangStability},
the dense-core method developed by Fang-Lin-Zhai \cite{FangLinZhai},
and the residual function introduced by Zhai-Li-Lou \cite{ZhaiLiLou}.

\section{Preliminaries}

We first introduce two notations needed to state the following edge-spectral stability result due to Li, Liu, and Zhang \cite{LiLiuZhangStability}.
Let $G$ and $H$ be two graphs (which may have distinct vertex sets).
Define the \emph{distance} between $G$ and $H$ as
$$d(G, H):=|E(G)\setminus E(H)|+|E(H)\setminus E(G)|,$$
which counts the minimum number of edge modifications (additions or deletions) needed to transform $H$ into $G$.
Given two disjoint vertex sets $U$ and $V$,
we use $K_{U,V}$ to represent the complete bipartite graph with parts $U$ and $V$.

\begin{lemma}[Li, Liu, and Zhang \cite{LiLiuZhangStability}]\label{thm:stability}
Let $F$ be fixed with $\chi(F)=3$.
For every $\varepsilon >0$, there exists a constant $\delta=\delta(F,\varepsilon)>0$ and a positive integer $m_0$ such that
the following holds.
If $G$ is an $F$-free graph with $m\geq m_0$ edges satisfying
$\rho(G)\geq\sqrt{(1-\delta)m},$
then there exist two disjoint sets $U,V\subseteq V(G)$ such that
$d(G,K_{U,V})\leq \eps m.$
Moreover, $\rho(G)\le\sqrt{(1+o(1))m}$ for every $m$-edge $F$-free graph $G$.
\end{lemma}

We now introduce another useful result established by
 Fang, Lin and Zhai (see \cite{FangLinZhai1, FangLinZhai}).
For a graph $G$ with $e(G)\geq1$, define
$\Phi(G):=\rho(G)\big/\sqrt{e(G)}.$
Fix $0<\eps<1/100$.
Following Fang--Lin--Zhai \cite{FangLinZhai}, an edge-induced subgraph $G'\subsetneq G$ is called
\emph{$\eps$-dense} in $G$ if both the following hold:
\vspace{1mm}

(i) $(1-\eps)e(G)<e(G')<e(G)$;

(ii) $\Phi(G')-\Phi(G)\ge \eps(e(G)-e(G'))\big/\big(2e(G)\big).$
\vspace{1mm}

\noindent An \emph{\(\varepsilon\)-core} is defined as a graph
containing no proper \(\varepsilon\)-dense subgraphs (and having no isolated vertices).
The following lemma gives two estimates for an $\eps$-core.

\begin{lemma}[Fang, Lin, and Zhai \cite{FangLinZhai}]\label{lem:core}
Let $H$ be an \(\varepsilon\)-core with sufficiently large size $h$,
and let \(\boldsymbol{x}=(x_v)\) be a nonnegative unit eigenvector for \(\rho(H)\).
If $\rho(H)\geq \sqrt{h}$, then

(i) $x_ux_v>(1-\eps)\big/\big(4\sqrt h\big)$ for every edge $uv\in E(H)$;

(ii) $2h x_u^2\ge(1-2\eps)d_H(u)$ for every vertex $u\in V(H)$ with $d_H(u)< \varepsilon h$.
\end{lemma}

These are the general dense-core estimates in
\cite[Lemmas 4.3 and 4.4]{FangLinZhai}.
Their proofs rely only on the definition of an \(\varepsilon\)-core and the Rayleigh quotient.
If a graph $G$ is disconnected, it may possess distinct nonnegative unit eigenvectors associated with \(\rho(G)\).
Nevertheless, there must exist a component \(G_0\) of $G$ satisfying \(\rho(G_0)=\rho(G)\).
By the Perron–Frobenius theorem, $G$ has a nonnegative unit eigenvector \(\boldsymbol{x}=(x_v)\)
such that the restriction of \(\boldsymbol{x}\) to \(V(G_0)\) is entrywise positive,
while all coordinates of \(\boldsymbol{x}\) corresponding to vertices outside \(V(G_0)\) are zero.
As a consequence, assertion (i) of Lemma \ref{lem:core} implies that
every \(\varepsilon\)-core $H$ is connected whenever \(e(H)\) is sufficiently large and \(\rho(H)\geq\sqrt{e(H)}\).

For fixed $k\ge2$, define
\[
 \vartheta_k(t):=
 \begin{cases}
  g^*(t),&k=2,\\
  g_k(t),&k\ge3.
 \end{cases}
\]
From \eqref{eq:gdef} and \eqref{eq:gasymp}, we know that $\vartheta_k(m)=\big(1+o(1)\big)\sqrt{m}$ for sufficiently large $m$.

We prove the following dense-core reduction lemma.

\begin{lemma}\label{lem:core-reduction}
Fix constants $0<\varepsilon<1/100$ and an integer $k\ge 2$.
Let $G$ be a $W_{2k+1}$-free graph of sufficiently large size $m$.
If $\rho(G)\geq\vartheta_k(m)$, then $G$ contains a connected
$\eps$-core $H$ satisfying $e(H)=(1-o(1))m$ and
$\rho(H)\geq\vartheta_k(e(H)).$
In particular, if \(\rho(H)=\vartheta_k(e(H))\), then $H=G$.
\end{lemma}

\begin{proof}
We begin by performing the iteration on $G$ to extract a desired \(\varepsilon\)-core $H$.
If $G$ itself is an $\eps$-core, there is nothing to prove,
since the connectedness of such an \(\varepsilon\)-core has been established earlier.
Now, $H=G$.
Otherwise, set $G_0:=G$ and $m_0:=m$.
Whenever $G_i$ is not an $\eps$-core, choose an $\eps$-dense subgraph
$G_{i+1}\subsetneq G_i$.  Write
$m_i=e(G_i)$ and $\Delta_i=m_i-m_{i+1}.$
By the definition of $\eps$-dense subgraphs, we know
$0<\Delta_i<\eps m_i$ and
\begin{equation}
        \Phi(G_{i+1})-\Phi(G_i)
        \ge \eps\Delta_i\big/(2m_i).
        \label{eq:core-stepgain}
\end{equation}

The number of edges is a strictly decreasing positive integer, so the
process terminates after some steps, say at $G_r=:H$.
Then $r\geq1$, $m_r<m$ and $m_i\leq m$ for all $i\in \{0,\ldots,r-1\}$. Since
$W_{2k+1}$-freeness is hereditary, $H$ is $W_{2k+1}$-free.
It suffices to show $e(H)=(1-o(1))m$ and
$\rho(H)>\vartheta_k(e(H)),$
since the connectedness of $H$ has already been established
whenever both the edge and spectral thresholds hold.

We first prove $e(H)=(1-o(1))m$.
Summing \eqref{eq:core-stepgain} over $i$ from $0$ to $r-1$ yields
\begin{equation}
 \Phi(H)-\Phi(G)
 \ge\frac{\eps}{2}\sum_{i=0}^{r-1}\frac{\Delta_i}{m_i}.
        \label{eq:core-totalgain}
\end{equation}
We first prove that $e(H)=m_r$ is still of linear order in $m$.  Set
$x_i=\Delta_i/m_i$.  Then we have $0<x_i<\eps$ and
$m_{i+1}=m_i(1-x_i)$. Thus,
\[
\log\frac {m_0}{m_r}=\sum_{i=0}^{r-1}\log\frac{m_i}{m_{i+1}}=\sum_{i=0}^{r-1}-\log(1-x_i).
\]
For $x\in [0,\eps]$, we have
$-\log(1-x)=\int_0^x\frac{dt}{1-t}\le\frac{x}{1-\eps}.$
Since $x_i\in (0,\eps)$, we obtain
\[
 \sum_{i=0}^{r-1}\frac{\Delta_i}{m_i}=\sum_{i=0}^{r-1}x_i\geq (1-\eps)\sum_{i=0}^{r-1}-\log(1-x_i)
=(1-\eps)\log\frac m{m_r}.
\]
Together with \eqref{eq:core-totalgain}, this gives
\begin{equation}
 \Phi(H)-\Phi(G)
 \ge\frac{\eps(1-\eps)}2\log\frac m{m_r}.
        \label{eq:core-loggain}
\end{equation}
We know that
$\rho(H)^2\le\operatorname{tr}(A(H)^2)=2e(H)$, and therefore
$\Phi(H)=\rho(H)/\sqrt{e(H)}\le\sqrt2$.  On the other hand, the assumption $\rho(G)\geq\vartheta_k(m)$ implies
$\Phi(G)\ge\vartheta_k(m)/\sqrt m=1+o(1)$.  Thus,
\eqref{eq:core-loggain} gives $\log(m/m_r)=O_{\eps}(1)$, and so
$m_r\ge c_{\eps}m$ for some constant $c_{\eps}>0$.

We next upgrade the inequality $m_r\ge c_{\eps}m$ to
\begin{equation}
       m_r=\big(1-o(1)\big)m.
        \label{eq:core-nearm}
\end{equation}
Indeed, if this failed, then $m_r\le(1-\alpha)m$ for some fixed $\alpha>0$.  Since $m_i\le m$,
\eqref{eq:core-totalgain} gives
\begin{equation}
 \Phi(H)-\Phi(G)
 \ge\frac{\eps}{2m}\sum_{i=0}^{r-1}\Delta_i
 =\frac{\eps(m-m_r)}{2m}
 \ge\frac{\eps\alpha}{2}.
        \label{eq:core-fixedgain}
\end{equation}
But $H$ is $W_{2k+1}$-free and $m_r\ge c_{\eps}m$, so
Lemma~\ref{thm:stability} gives $\Phi(H)\le1+o(1)$, whereas
$\Phi(G)\ge1+o(1)$.  This contradicts \eqref{eq:core-fixedgain}.
Hence, \eqref{eq:core-nearm} holds.

Next, we show that $\rho(H)\geq\vartheta_k(e(H)).$  Define
$\Psi_k(t):=\frac{\vartheta_k(t)}{\sqrt t}.$
If $k=2$, then
$\Psi_2(t)=\frac{g^*(t)}{\sqrt t}=\frac1{2\sqrt t}+\sqrt{1+\frac1{4t}},$
and hence
$\Psi_2'(t)=-\frac1{4t^{3/2}}-\frac1{8t^2\sqrt{1+1/(4t)}}.$
If $k\ge3$, then $\Psi_k(t)=\frac{g_k(t)}{\sqrt t}=\frac{k-1}{2\sqrt t}+\sqrt{1-\frac{k^2-1}{4t}},$
and so
\[
 \Psi_k'(t)=-\frac{k-1}{4t^{3/2}}
 +\frac{k^2-1}{8t^2\sqrt{1-(k^2-1)/(4t)}}.
\]
Thus, in both cases, there is a constant $C_k$ such that
\begin{equation}
        \big|\Psi_k'(t)\big|\le C_kt^{-3/2}
        \label{eq:core-derivative}
\end{equation}
for all sufficiently large $t$.

By \eqref{eq:core-nearm}, we may assume $m_r\ge m/2$.
The mean value theorem gives
$\big|\Psi_k(m_r)-\Psi_k(m)\big|=\big|\Psi_k'(t_0)\big|(m-m_r)$
for some $t_0\in (m_r,m)\subseteq (m/2,m).$
Combining this with \eqref{eq:core-derivative} yields
\begin{equation}
 \big|\Psi_k(m_r)-\Psi_k(m)\big|\leq2\sqrt{2}C_km^{-3/2}(m-m_r).
\label{eq:core-drift}
\end{equation}
Meanwhile, from \eqref{eq:core-totalgain} and $m_i\le m$ for $i\in \{0,\ldots,r-1\}$,
we know that
\begin{equation}
 \Phi(H)-\Phi(G)\ge\frac{\eps}{2m}\sum_{i=0}^{r-1}\Delta_i
=\frac\eps 2m^{-1}(m-m_r).
\label{eq:core-finalgain}
\end{equation}
Since $m_r<m$, the right-hand side of \eqref{eq:core-finalgain} is strictly greater than that of \eqref{eq:core-drift}.
Moreover, the condition $\rho(G)\geq\vartheta_k(m)$ gives $\Phi(G)\geq\Psi_k(m).$
Consequently,
\begin{equation*}
 \Phi(H)\!-\!\Psi_k(m_r)
 =\bigl(\Phi(G)\!-\!\Psi_k(m)\bigr)\!+\!\bigl(\Phi(H)\!-\!\Phi(G)\bigr)
 \!-\!\bigl(\Psi_k(m_r)\!-\!\Psi_k(m)\bigr)>0.
\end{equation*}
We therefore conclude that \(\rho(H)>\vartheta_k(e(H))\), since $e(H)=m_r$.
\end{proof}

Let $G$ be a connected graph.
By the Perron–Frobenius theorem,
there exists a positive unit eigenvector corresponding to $\rho(G)$,
which is called the \emph{Perron vector} of $G$.
The lemma below concerns a residual function $f(w)$.
This function was introduced by Zhai, Li, and Lou \cite{ZhaiLiLou} to precisely characterize extremal graph structures.

\begin{lemma}\label{lem:residual}
Let $G$ be a connected graph with spectral radius $\rho$ and Perron vector \(\boldsymbol{x}=(x_v)\).
Choose $u^*$ with
$x_{u^*}=\max_{v\in V(G)}x_v$. Define
$U=N_G(u^*)$,
$W=V(G)\setminus(\{u^*\}\cup U)$,
and
$f(w)=d_U(w)(x_{u^*}\!-\!x_w)+\frac12d_W(w)x_{u^*}$
for any $w\in W$.
For arbitrary real constants $c$ and $d$, we have
\begin{equation}\label{eq:resid2}
\bigl(\rho^2\!-\!d\rho\!+\!c\!-\!e(G)\bigr)x_{u^*}
=\!\!\!\sum_{uv\in E(U)}\!\!\!(x_u\!+\!x_v\!-\!x_{u^*})
\!-\!(d\rho\!-\!c)x_{u^*}\!-\!\sum_{w\in W}\!f(w).
\end{equation}
\end{lemma}

\begin{proof}
By the eigen-equation at $u^*$, we derive
$\rho^2x_{u^*}=\sum_{v\in U}\sum_{u\in N(v)}x_u.$
This further yields
\begin{equation*}
\rho^2x_{u^*}
=|U|x_{u^*}
\!+\!\!\sum_{uv\in E(U)}\!\!(x_u\!+\!x_v)
\!+\!\sum_{w\in W}d_U(w)x_w.
\end{equation*}
Since $e(U,W)=\sum_{w\in W}d_U(w)$ and $e(W)=\frac12\sum_{w\in W}d_W(w)$, we may rewrite
\[
\begin{aligned}
|U|x_{u^*}\!+\!\sum_{w\in W}d_U(w)x_w
  =\big(|U|\!+\!e(U,W)\!+\!e(W)\big)x_{u^*}\!-\!\sum_{w\in W}\!f(w).
\end{aligned}
\]
Note that $|U|+e(U,W)+e(W)=e(G)-e(U)$. It follows that
\begin{equation}\label{eq:resid1}
\rho^2x_{u^*}=e(G)x_{u^*}\!+\!\!\sum_{uv\in E(U)}\!\!(x_u\!+\!x_v\!-\!x_{u^*})\!-\!\sum_{w\in W}\!f(w).
\end{equation}
Subtracting $(d\rho-c+e(G))x_{u^*}$ from both sides of \eqref{eq:resid1},
we obtain the desired equality \eqref{eq:resid2}.
\end{proof}

\section{Sharpness of Theorems \ref{thm1} and \ref{thm2}}\label{sec3}

The following three propositions establish the sharpness of our main theorems.

\begin{proposition}\label{prop:w5}
Let $n$ be even, and let $G_n$ be a graph of size $m=n^2+n$ obtained from $K_{n,n}$ by inserting
a perfect matching in each part.  Then $G_n$ is $W_5$-free.
Moreover, we have $\rho(G_n)=g^*(m)$
and $\rho(G_n)^2-\rho(G_n)=m.$
\end{proposition}

\begin{proof}
The graph $G_n$ is clearly $(n+1)$-regular, so $\rho(G_n)=n+1$.
Since $m=n^2+n$, we have $\rho(G_n)^2-\rho(G_n)=m$.
This further implies that $\rho(G_n)=\frac12\big(1+\sqrt{4m+1}\big)=g^*(m)$.

If $v$ belongs to one part and $v'$ is
its matching partner, then $G_n[N(v)]$ is the join of $v'$ with a
matching in the opposite part, and hence contains no $4$-cycles.
\end{proof}

\begin{proposition}\label{prop:w2k+1}
Let $G=K_k\vee qK_1$,
where $q$ is a positive integer such that ${k \choose 2}+kq=m$.
Then $G$ is $W_{2k+1}$-free. Moreover, we have $\rho(G)=g_k(m)$ and
\begin{equation}\label{eq:thresholdidentity}
\rho(G)^2-(k-1)\rho(G)=m-{k \choose 2}.
\end{equation}
%which implies that
%\begin{equation}\label{eq:gasymp}
%g_k(m)=\sqrt m+\frac{k-1}{2}+O_k(m^{-1/2}).
%\end{equation}
\end{proposition}

\begin{proof}
For every vertex $v\in V(G)$,
the longest cycle contained in the induced subgraph $G[N_G(v)]$ has length less than $2k$.
Thus, $G$ is $W_{2k+1}$-free.
Let $\rho=\rho(G)$, and let $\bm x=(x_v)$ be the non-negative unit eigenvector corresponding to $\rho(G)$.
By the symmetry of $G$,
for any vertex $u$ belonging to the $k$-clique and any vertex $w$ lying in the independent set of size $q$,
we have $\rho x_u=(k-1)x_u+qx_w$ and $\rho x_w=kx_u$.
Substituting the relation ${k \choose 2}+kq=m$ into the above two equations yields identity \eqref{eq:thresholdidentity}.
Solving this system further gives $\rho(G)=\frac12\big(k-1+\sqrt{4m-k^2+1}\big)=g_k(m)$.
\end{proof}

\begin{proposition}\label{lem:w5-matching}
Let $G$ have a partition $V(G)=S\cup T$ such that
$\Delta(G[S])\le1$ and $\Delta(G[T])\le1.$
Then $\rho(G)^2-\rho(G)\le e(G).$
If equality holds, then there exists an
even integer $n$ such that
$G\cong K_{n,n}+M_S+M_T,$
where $M_S,M_T$ are perfect matchings embedded in $S$ and $T$ respectively.
\end{proposition}

\begin{proof}
Consider a component $G_0$ with $\rho(G_0)=\rho(G)$.
Since $e(G_0)\le e(G)$, proving the inequality for $G_0$
proves it for $G$.  If equality holds for $G$, then necessarily
$e(G_0)=e(G)$, so every other component is isolated.  We may therefore
assume that $G$ is connected.

Denote $q=e(S,T),$ $r_S=e(S),$ $r_T=e(T),$
and write $\rho=\rho(G)$.
The connectedness of $G$ guarantees that $G$ admits a Perron vector \(\boldsymbol{x}=(x_v)\).
Define
$a=\max_{u\in S}x_u$ and $b=\max_{v\in T}x_v.$
Choose $u\in S$ with $x_u=a$.  Since $u$ has at most one neighbor in
$S$ and every Perron coordinate is at most $a$ on $S$, we have
\[
 \sum_{v\in N_T(u)}x_v
 =\rho x_u-\sum_{w\in N_S(u)}x_w\ge\big(\rho-1\big)a.
\]
Multiplying by $\rho$ and expanding at the vertices of $N_T(u)$ gives
$\rho\sum_{v\in N_T(u)}x_v
 =\sum_{v\in N_T(u)}\sum_{w\in N(v)}x_w.$
The contributions $x_w$ from $w\in S$ use at most $q$ cross edges, each with
coordinate at most $a$.  The contributions $x_w$ from $w\in T$ use at most
$2r_T$ endpoints of the matching in $G[T]$, each with coordinate at most
$b$. Combining the above two inequalities yields
$\rho(\rho-1)a\le qa+2r_Tb,$ i.e.,
\begin{equation}
        \big(\rho(\rho-1)-q\big)a\le 2r_Tb.
        \label{eq:w5-match1}
\end{equation}
By symmetry, we have
\begin{equation}
  \big(\rho(\rho-1)-q\big)b\le 2r_Sa.
        \label{eq:w5-match2}
\end{equation}
Let $D_0=\rho(\rho-1)-q$.  If $D_0\le0$, then
$\rho^2-\rho\le q\le e(G)$.  If $D_0>0$, multiplying the preceding
inequalities gives
$D_0^2\le4r_Sr_T.$
Therefore, we also have
\[
 \rho^2-\rho=q+D_0
 \le q+2\sqrt{r_Sr_T}
 \le q+r_S+r_T=e(G).
\]

Suppose equality $\rho^2-\rho=e(G)$ holds. Then \(D_0\geq0\), and equality must hold at all preceding steps.
If $D_0=0$, then $q=e(G)$ and $G$ is bipartite. Consequently,
$\rho(G)^2\leq e(G)$, which contradicts $\rho^2-\rho=e(G)$. Thus $D_0>0$.
In particular, \(r_S=r_T\), and equality holds in both \eqref{eq:w5-match1} and \eqref{eq:w5-match2}.
Equality in the first Perron estimate implies that the selected maximizer vertex \(u\in S\) possesses exactly one neighbor inside $S$,
and this neighbor carries coordinate $a$.
Moreover, for every cross edge, its endpoint in $T$ lies in $N_T(u)$, while its endpoint in $S$ has Perron coordinate $a$. Similarly, every endpoint of an edge in $G[T]$ lies in $N_T(u)$ and has Perron coordinate $b$. Since $G$ has no isolated vertices, it follows that $N_T(u)=T$.
By symmetry, any vertex attaining the maximum Perron coordinate in $T$ is adjacent to every vertex in $S$.
Furthermore, since \(r_S=r_T\), the equality conditions in \eqref{eq:w5-match1} and \eqref{eq:w5-match2} together entail \(a=b\).
Repeating the foregoing saturation reasoning shows that every vertex incident to an edge bears the identical Perron coordinate \(a=b\).
Connectedness of $G$ consequently forces the Perron vector to be constant over all vertices, implying that $G$ is regular.

Write $|S|=s$ and $|T|=t$.  The maximum vertex in $S$ has degree
$t+1=\rho$, while the maximum vertex in $T$ has degree $s+1=\rho$.
Thus $s=t=:n$.  Every vertex has degree $n+1$, but it has at most $n$
cross neighbors and at most one same-side neighbor.  Hence, every cross
edge is present and every vertex has exactly one same-side neighbor.
Thus, the cross graph is $K_{n,n}$ and the two internal graphs are
perfect matchings.
\end{proof}

\section{Stability partition on dense cores}\label{sec4}

Throughout this section, \(k\ge 2\) is fixed.
We select constants satisfying the hierarchy
\begin{equation}\label{eq:parameter-hierarchy}
0<\eps_0\ll\eta^2\ll\eta\ll 1.
\end{equation}
Additionally, if \(k\ge 3\), we further require \(\eta\ll 1/k\).
All asymptotic statements are taken as $m\to\infty$ with
$k,\varepsilon_0$ and $\eta$ fixed.

The dense-core reduction is always applied with $\eps_0$.
By Lemma \ref{lem:core-reduction}, the contradiction arguments in later sections may be reduced to
a connected \(W_{2k+1}\)-free \(\varepsilon_0\)-core graph $G$ with a sufficiently large number of edges $m$ such that
\begin{equation}
       \rho(G)\geq\vartheta_k(m),
        \label{eq:core-rhosq}
\end{equation}
where $\vartheta_2(m)=g^*(m)$ and $\vartheta_k(m)=g_k(m)$ for every fixed $k\geq3$.
Combining \eqref{eq:core-rhosq} with \eqref{eq:gdef} and \eqref{eq:gasymp}, we know that $\rho(G)\geq\vartheta_k(m)>\sqrt{m}$.

On the other hand,
Lemma~\ref{thm:stability} gives that \(\rho(G)^2=(1+o(1))m\),
and there exist two disjoint vertex subsets $A,B\subseteq V(G)$ such that
\begin{equation}
        D:=d(G,K_{A,B})=o(m).
        \label{eq:D}
\end{equation}
Among all such pairs, choose $A,B$ first so that $D$ is minimum and, subject
to this, so that $|R|$ is minimum, where $R=V(G)\setminus(A\cup B)$.
Without loss of generality, write $a=|A|\le b=|B|.$
Since every discrepancy between the $m$ edges of $G$ and the $ab$ edges of $K_{A,B}$ belongs to
the symmetric difference, we have $|m-ab|\le D.$
Combining this with \eqref{eq:D} gives
\begin{equation}
        ab=\big(1+o(1)\big)m.
        \label{eq:abm}
\end{equation}
For any $w\in V(G)$, let $N_A(w)=N_G(w)\cap A$ and $d_A(w)=|N_A(w)|$.
The minimality of $D$ and $|R|$ implies that
\begin{equation}
\begin{split}
 d_B(u)&\ge b/2 \quad(u\in A),\\
 d_A(v)&\ge a/2 \quad(v\in B),\\
 d_A(w)&< a/2,~~ d_B(w)< b/2 \quad(w\in R).
\end{split}
\label{eq:minedit}
\end{equation}
Indeed, if $u\in A$ had $d_B(u)<b/2$, removing $u$ from $A$ would
replace $b-d_B(u)$ missing cross edges by only $d_B(u)$ extra edges,
strictly decreasing $D$.  The other assertions follow in the same
way by moving a vertex into or out of a part.

For every fixed $\gamma>0$, define two good-crossing vertex subsets as follows:
$A_\gamma=\big\{u\in A:d_B(u)\ge(1-\gamma)b\big\}$ and
$B_\gamma=\big\{v\in B:d_A(v)\ge(1-\gamma)a\big\}.$
The number of missing $A$--$B$ edges is at most $D=o(m)$.
Combining this with \eqref{eq:abm} yields
\begin{equation}
 |A\setminus A_{\gamma}|=o(a)  \quad  \mbox{and} \quad
 |B\setminus B_{\gamma}|=o(b)
        \label{eq:goodsize}
\end{equation}
for any fixed constant $\gamma>0$. Notice that no assumption $a,b\to\infty$ is needed in \eqref{eq:goodsize}.
In what follows, we write $A^*=A_{\eta^2}$ and $B^*=B_{\eta^2}$ for simplicity.

\begin{lemma}\label{lem:small}
For every $w\in V(G)$, if $d_B(w)>4\eta^2b$, then
$d_A(w)\le \eta^2a+k-1$.  If in addition $a\to\infty$, then the
symmetric assertion also holds: $d_A(w)>4\eta^2a$ implies
$d_B(w)\le\eta^2b+k-1$.  Consequently, when $a\to\infty$,
$d_A(u)\le2\eta^2a$ for every $u\in A$ and
$d_B(v)\le2\eta^2b$ for every $v\in B$.
\end{lemma}

\begin{proof}
Suppose that $d_B(w)>4\eta^2b$ and
$d_A(w)>\eta^2a+k-1$.  Since $|A\setminus A^*|=o(a)$ by \eqref{eq:goodsize}, for all
sufficiently large $m$ the set $N_A(w)\cap A^*$ contains $k$ distinct
vertices $v_1,\ldots,v_k$, with indices read modulo $k$.  Each $v_i$ misses
at most $\eta^2b$ vertices of $B$, and hence
\[
 |N_B(w)\cap N_B(v_i)\cap N_B(v_{i+1})|
 \ge d_B(w)-2\eta^2b>2\eta^2b.
\]
Since $\eta$ is constant and $b\to\infty$ by \eqref{eq:abm}, we may choose distinct
$u_i\in N_B(w)\cap N_B(v_i)\cap N_B(v_{i+1})$.  Then,
$v_1u_1v_2u_2\cdots v_ku_kv_1$
forms a $2k$-cycle contained in $G[N_G(w)]$.
This gives a copy of $W_{2k+1}$ centered at
$w$, a contradiction. This proves the first assertion.

If $a\to\infty$, the same argument with the roles of $A$ and $B$ reversed
proves the symmetric assertion.  Finally, if $u\in A$, then
$d_B(u)\ge b/2>4\eta^2b$ by \eqref{eq:minedit}; hence $d_A(u)\le\eta^2a+k-1\le2\eta^2a$ for
large $m$.  The assertion for vertices of $B$ follows symmetrically.
\end{proof}

Since the $\varepsilon_0$-core $G$ is connected, let $\bm x=(x_v)$ denote its Perron vector.

\begin{lemma}\label{lem:perron-mass}
There is a constant $C>0$ such that, for all sufficiently large $m$,
\[
 x_u^2\ge\big(1-C\eta^2\big)\frac{b}{2m}~~(u\in A^*)
 \quad \mbox{and}\quad
 x_v^2\ge\big(1-C\eta^2\big)\frac{a}{2m}~~(v\in B^*).
\]
Consequently, both
$\sum_{u\in A^*}x_u^2$ and $\sum_{v\in B^*}x_v^2$ are bounded below by $\frac12\big(1-C\eta^2-o(1)\big).$
In particular,
for all $S\subseteq A$ and $T\subseteq B$,
\[
 \sum_{u\in S}x_u^2\le\frac{|S|}{2a}+C\eta^2+o(1)
\quad \mbox{and}\quad
 \sum_{v\in T}x_v^2\le\frac{|T|}{2b}+C\eta^2+o(1).
\]
\end{lemma}

\begin{proof}
Take $v\in B^*$. Then $d_A(v)\ge(1-\eta^2)a$.
Since $a^2\le ab=(1+o(1))m$, it follows that
$a=O(\sqrt m)=o(m)$.  Recall also from \eqref{eq:parameter-hierarchy} that $\varepsilon_0\ll \eta^2\ll1$.
Every neighbor of $v$ outside $A$ is incident to
an edit-error edge, so $d_G(v)\le a+D<\varepsilon_0m$ for large $m$.
Consequently, conclusion (ii) of Lemma \ref{lem:core} gives
\begin{equation}\label{lower1}
 x_v^2\ge\big(1-2\varepsilon_0\big)\frac{d_G(v)}{2m}\geq\big(1-2\varepsilon_0\big)\big(1-\eta^2\big)\frac{a}{2m}
 \ge\big(1-3\eta^2\big)\frac{a}{2m}.
\end{equation}

Now take $u\in A^*$. Then $d_B(u)\ge(1-\eta^2)b$.
Since $|B\setminus B^*|=o(b)$ by \eqref{eq:goodsize},
we have $d_{B^*}(u)\ge(1-2\eta^2)b$.
Applying the lower bound from \eqref{lower1} to each $v\in N_{B^*}(u)$ and invoking the eigen-equation at $u$,
we arrive at
$$\rho x_u\geq\!\!\sum_{v\in N_{B^*}(u)}\!\!x_v\ge\big(1-2\eta^2\big)b\sqrt{\big(1-3\eta^2\big)\frac{a}{2m}}.$$
Squaring and using \eqref{eq:abm} together with $\rho^2=(1+o(1))m$, we further obtain
\begin{equation}\label{lower2}
x_u^2\ge\big(1-C\eta^2\big)\frac{b}{2m}
\end{equation}
for some constant $C\geq3$.
Combining \eqref{lower1} and \eqref{lower2} establishes the first assertion.

Summing \eqref{lower1} over $B^*$ and \eqref{lower2} over $A^*$, and using the estimates
$|A^*|=(1-o(1))a$, $|B^*|=(1-o(1))b$ together with $ab=(1+o(1))m$, we obtain
$$\min\Big\{\sum_{u\in A^*}\!x_u^2,\sum_{v\in B^*}\!x_v^2\Big\}
\geq \big(1-C\eta^2\big)\big(1-o(1)\big)\frac{ab}{2m}
\geq\frac12\big(1-C\eta^2-o(1)\big).$$

Finally, choose $S\subseteq A$ arbitrarily.
Since $\sum_{w\in V(G)}x_w^2=1$, we have
\begin{equation}\label{lower4}
\sum_{u\in A}x_u^2\le1-\sum_{v\in B^*}x_v^2
\leq\frac12\big(1+C\eta^2+o(1)\big).
\end{equation}

Set $\lambda_{A^*}=(1-C\eta^2)b\big/(2m)$.
Recall $ab=(1+o(1))m$.
Then $\lambda_{A^*}=\frac1{2a}\big(1-C\eta^2+o(1)\big),$
and hence $\lambda_{A^*}|S|\leq|S|/(2a)+o(1)$.
Moreover, recall from \eqref{eq:goodsize} that $|A^*|=(1-o(1))a$,
which implies
$\lambda_{A^*}|A^*|=\frac12\big(1-C\eta^2-o(1)\big)$.
The lower bound \eqref{lower2} then yields
\begin{equation}\label{lower5}
\sum_{u\in A^*\setminus S}x_u^2\geq\lambda_{A^*}|A^*\setminus S|\geq\frac12\big(1-C\eta^2-o(1)\big)-\frac{|S|}{2a}.
\end{equation}

Combining \eqref{lower4} and \eqref{lower5}, we obtain
$$\sum_{u\in S}x_u^2=\sum_{u\in A}x_u^2-\sum_{u\in A\setminus S}\!\!x_u^2
\leq\sum_{u\in A}x_u^2-\sum_{u\in A^*\setminus S}\!\!x_u^2\leq\frac{|S|}{2a}+C\eta^2+o(1).$$
The proof for $T\subseteq B$ is identical. This proves the lemma.
\end{proof}

\begin{lemma}\label{lem:R-empty}
The minimum-edit partition $R$ is empty.
\end{lemma}

\begin{proof}
Suppose that $w_0\in R$.  Since every edge incident to $w_0$ is an edit error,
$d_G(w_0)\le D=o(m)$, so conclusion (ii) of Lemma \ref{lem:core} applies.
Thus, $x_{w_0}^2\ge(1-2\varepsilon_0)d_G(w_0)/2m.$
By the Cauchy--Schwarz inequality and then the eigen-equation at $w_0$,
we obtain
$$\sum_{w\in N(w_0)}x_w^2\geq \frac1{d_G(w_0)}\Big(\sum_{w\in N(w_0)}\!\!x_w\Big)^2=\frac{\rho^2x_{w_0}^2}{d_G(w_0)}
\geq\big(1-2\varepsilon_0\big)\frac{\rho^2}{2m}.$$
Recall from \eqref{eq:core-rhosq} that $\rho^2>m$.
It follows that $\sum_{w\in N(w_0)}x_w^2\geq(1-2\varepsilon_0)/2.$

We now derive a smaller upper bound.  First suppose $d_B(w_0)\le4\eta^2b$.
Note that $\sum_{w\in V(G)}x_w^2=1$.
By Lemma \ref{lem:perron-mass},
\begin{equation}\label{lower3}
\sum_{w\in R}x_w^2\le1-\big(1-C\eta^2-o(1)\big)=C\eta^2+o(1).
\end{equation}
Furthermore, choose $S=N_A(w_0)$ and $T=N_B(w_0)$.
Again by Lemma \ref{lem:perron-mass}, combined with $d_A(w_0)<a/2$ from \eqref{eq:minedit}
and the assumption $d_B(w_0)\le4\eta^2b$,
we obtain
\[
 \sum_{u\in N_A(w_0)}\!\!x_u^2+\!\!\sum_{v\in N_B(w_0)}\!\!x_v^2
 \le \frac{d_A(w_0)}{2a}+\frac{d_B(w_0)}{2b}+2C\eta^2+o(1)\leq\frac 14+O(\eta^2)+o(1).
\]
Combining this with \eqref{lower3} yields
$\sum_{w\in N(w_0)}x_w^2\le\frac14+O(\eta^2)+o(1)$,
contradicting the earlier lower bound for $\sum_{w\in N(w_0)}x_w^2$.
Hence, $d_B(w_0)>4\eta^2b$.

Since $d_B(w_0)>4\eta^2b$, Lemma \ref{lem:small} implies
$d_A(w_0)\le\eta^2a+k-1$.  If $a\ge2k$, then
$d_A(w_0)/(2a)\le\eta^2/2 +1/4-1/(4k)$, and
$d_B(w_0)/(2b)<1/4$ by \eqref{eq:minedit}.
Combining this with \eqref{lower3}, we have
\[
 \sum_{w\in N(w_0)}x_w^2
 \le\frac12-\frac1{4k}+O(\eta^2)+o(1),
\]
which again contradicts $\sum_{w\in N(w_0)}x_w^2\geq(1-2\varepsilon_0)/2.$

It remains only to consider $a<2k$.  Since $d_A(w_0)<a/2$ by \eqref{eq:minedit}, and $a, d_A(w_0)$ are
integers, we have $d_A(w_0)/(2a)\le1/4-1/(4a)\le1/4-1/(8k)$.
Applying the same subset estimate yields
\[
 \sum_{w\in N(w_0)}x_w^2
 \le\frac12-\frac1{8k}+O(\eta^2)+o(1),
\]
which is also strictly smaller than $(1-2\varepsilon_0)/2$.
Therefore, no vertex in $R$ exists.
\end{proof}

From Lemma \ref{lem:R-empty}, we know that $V(G)=A\cup B$.  Define
$$x_A^*=\max_{u\in A}x_u, ~~x_B^*=\max_{v\in B}x_v, ~~
x^*=\max\{x_A^*,x_B^*\}.$$
Choose $w\in V(G)$ arbitrarily.
Applying the two-step eigen-equation gives
$\rho^2x_{w}=\sum_{u\in N(w)}\rho x_u=\sum_{u\in N(w)}\sum_{v\in N(u)}x_v.$
Observe that every edge other than an $A$--$B$ edge belongs to the edit-error set and is counted at most twice in the degree sums above.
Recall also from \eqref{eq:D} that $D=o(m)$.
Combining these observations, we arrive at
\begin{equation}\label{ooooo}
\rho^2x_{w}\leq \!\!\!\sum_{u\in N_A(w)}\!\!\!\!d_B(u)x^*_B+\!\!\!\sum_{u\in N_B(w)}\!\!\!\!d_A(u)x^*_A+2Dx^*
\leq d_A(w)bx^*_B+d_B(w)ax^*_A+o(m)x^*
\end{equation}
for every $w\in V(G)$. Denote $\alpha_w:=d_A(w)/a$ and $\beta_w:=d_B(w)/b.$
Then we have
\begin{equation}\label{29}
\rho^2x_w\leq ab\,\big(\alpha_w x^*_B+\beta_w x^*_A\big)+o(m)x^*.
\end{equation}
We shall use \eqref{ooooo} and \eqref{29} repeatedly.

\begin{lemma}\label{lem:partwise-perron-cap}
Suppose that $a\to\infty$.  Then
\begin{equation}\label{eq:partwise-perron-cap}
\begin{aligned}
 x_A^*&\le\big(1+O(\eta^2)+o(1)\big)\sqrt{\frac{b}{2m}}
          +\big(O(\eta^2)+o(1)\big)x^*,\\
 x_B^*&\le\big(1+O(\eta^2)+o(1)\big)\sqrt{\frac{a}{2m}}
          +\big(O(\eta^2)+o(1)\big)x^*.
\end{aligned}
\end{equation}
\end{lemma}

\begin{proof}
Since $a\to\infty$, the assumption $a\le b$ and \eqref{eq:abm} imply $a+b=o(m)$.
Lemma~\ref{lem:small} gives $d_A(u)\le2\eta^2a$ for $u\in A$ and
$d_B(v)\le2\eta^2b$ for $v\in B$.

Choose $u_0\in A^*$ to be a vertex attaining the minimum Perron coordinate over  $A^*$.
Then $x_{u_0}^2\leqq\frac1{|A^*|}\sum_{u\in A^*}x_{u}^2\leq\frac1{|A^*|}\sum_{u\in A}x_{u}^2.$
Recall from \eqref{eq:abm} that $ab=(1+o(1))m$,
and from \eqref{eq:goodsize} that $|A^*|=(1-o(1))a$.
Combining this with \eqref{lower4} yields
\begin{equation}\label{ggg}
 x_{u_0}\le\big(1+O(\eta^2)+o(1)\big)\sqrt{\frac{b}{2m}}.
\end{equation}
Choose $u^*\in A$ such that $x_{u^*}=x^*_A$.
Subtracting the eigen-equations at $u^*$ and $u_0$ gives
$\rho(x_{u^*}-x_{u_0}) \le\sum_{w\in N(u^*)\setminus N(u_0)}x_w$.
Applying the eigen-equation once more yields
\[
 \rho^2(x_{u^*}-x_{u_0})\leq\sum_{w\in N(u^*)\setminus N(u_0)}\!\!\rho x_w
 \le\sum_{w\in N(u^*)\setminus N(u_0)}\!\! d_G(w)x^*.
\]
We now estimate the degree sum $\sum_{w\in N(u^*)\setminus N(u_0)}d_G(w)$. Its regular $A$--$B$
incidences are at most
$d_A(u^*)\,b+|N_B(u^*)\setminus N_B(u_0)|\,a$.
Since \(u_0\in A^*\), we have \(d_B(u_0)\ge(1-\eta^2)b\).
The vertices in \(N_B(u^*)\setminus N_B(u_0)\) belong to the set of vertices in $B$ not adjacent to \(u_0\),
whose size is thus at most \(\eta^2 b\).
Meanwhile, Lemma \ref{lem:small} yields \(d_A(u^*)\le 2\eta^2 a\).
Consequently,
$d_A(u^*)\,b+|N_B(u^*)\setminus N_B(u_0)|\,a\leq3\eta^2ab$.
On the other hand, every remaining incidence comes from an edit-error edge, and each such edge is counted at
most twice. Therefore,
$\rho^2(x_{u^*}-x_{u_0})\le(3\eta^2ab+2D)x^*.$
Combining $\rho^2>m$, $ab=\big(1+o(1)\big)m$, and $D=o(m)$
with \eqref{ggg}, we obtain the first inequality in
\eqref{eq:partwise-perron-cap}.
The second follows by symmetry.
\end{proof}

\begin{lemma}\label{cor:bounded-ratio-perron-cap}
Suppose that $1\le b/a<100$.  Then
\begin{equation}\label{eq:balanced-cap}
 x_A^*\le\big(1+O(\eta^2)+o(1)\big)\sqrt{\frac{b}{2m}}
 ~~ \mbox{and} ~~
 x_B^*\le\big(1+O(\eta^2)+o(1)\big)\sqrt{\frac{a}{2m}}.
\end{equation}
Moreover,
$d_B(u)>(1-\eta)b$ for each $u\in A$ and
$d_A(v)>(1-\eta)a$ for each $v\in B$.
\end{lemma}

\begin{proof}
We first prove \eqref{eq:balanced-cap}.
Recall that $x^*=\max\{x_A^*,x_B^*\}$,
and apply Lemma \ref{lem:partwise-perron-cap}. If $x^*=x_A^*$, the first
inequality in \eqref{eq:partwise-perron-cap} absorbs its
$\big(O(\eta^2)+o(1)\big)x^*$ term and gives
$x_A^*\le(1+O(\eta^2)+o(1))\sqrt{b/(2m)}$.
Substitution in the second inequality in \eqref{eq:partwise-perron-cap},
together with $b/a<100$, gives the asserted bound on $x_B^*$.
If $x^*=x_B^*$, the symmetric argument gives the same two bounds.
This proves \eqref{eq:balanced-cap}.

Since \(ab=\bigl(1+o(1)\bigr)m\) and \(1\leq b/a<100\),
it follows that $a,b=\Theta(\sqrt m)$.
Then $d_G(u)\le a+b=\Theta(\sqrt m)<\varepsilon_0m$ for any $u\in A$.
In view of \eqref{eq:balanced-cap}, we obtain
\begin{equation}\label{sssss}
x^*=\max\big\{x_A^*,x_B^*\big\}=O\big(m^{-1/4}\big).
\end{equation}

Choose $u\in A$ arbitrarily. By Lemma \ref{lem:core},
$x_u^2\ge\big(1-2\varepsilon_0\big)d_G(u)\big/(2m).$
Recall that $\alpha_u=d_A(u)/a$ and $\beta_u=d_B(u)/b$.
Observe that $\sqrt{1-2\varepsilon_0}>1-2\varepsilon_0$ and $d_G(u)\geq d_B(u)=\beta_ub$.
Consequently,
\[ x_u\ge\sqrt{\big(1-2\varepsilon_0\big)\frac{d_G(u)}{2m}}\ge\big(1-2\varepsilon_0\big)\,\sqrt{\beta_u}\sqrt{\frac{b}{2m}}.
\]
On the other hand, inequality \eqref{29} establishes
$\rho^2x_u\leq ab\,\big(\beta_u x^*_A+\alpha_u x^*_B\big)+o(m)x^*,$
where $\alpha_u\leq 2\eta^2$ by Lemma \ref{lem:small}.
Thus
$\rho^2x_u\le ab(\beta_ux^*_A+2\eta^2x^*_B)+o(m)x^*.$
Combining $\rho^2>m$ and $ab=\big(1+o(1)\big)m$
with \eqref{eq:balanced-cap} and \eqref{sssss}, we obtain
\[
 x_u\le \beta_ux^*_A+2\eta^2x^*_B+o(1)x^*\leq\bigl(\beta_u+O(\eta^2)+o(1)\bigr)\sqrt{b/(2m)}.
\]
In view of \eqref{eq:minedit}, $\beta_u\geq1/2.$
Note that the function $h(t):=\sqrt t-t$ is decreasing when $t\geq1/2$.
If $\beta_u\le1-\eta$, then
$\sqrt{\beta_u}-\beta_u
 \ge\sqrt{1-\eta}-(1-\eta)
 >\frac{\eta}{3}$
for all $\eta\ll 1$. Since $\varepsilon_0\ll\eta^2\ll1$ and
$O(\eta^2)=o(\eta)$, the preceding two inequalities for $x_u$ are incompatible whenever $m$ is sufficiently large. Hence,
$\beta_u>1-\eta$, and therefore $d_B(u)>(1-\eta)b$ for each $u\in A$. The argument for
$v\in B$ is similar and hence omitted here.
\end{proof}

\begin{lemma}\label{lem:split-perron-hierarchy}
Suppose $b\ge100a$, and let $u^*$ satisfy $x_{u^*}=x^*$.
Then the following hold:

(i) $u^*\in A$ and $d_B(u^*)\ge(1-3\eta^2)\,b$;

(ii) $x_v\le\frac14x_{u^*}$ for all $v\in B$;

(iii) $d_B(u)>\frac{11}{20}\,b$ for each $u\in A$ satisfying $x_u>\frac35x_{u^*}$.
\end{lemma}

\begin{proof}
Passing to a subsequence if necessary, we may assume that either \(a\to\infty\) or $a$ remains bounded as $m\to\infty$.
In the latter case, since $k$ is fixed, we may write $a=O_k(1)$.
Firstly, we treat the case $a\to\infty$.
In this case, we must have $u^*\in A$, i.e., \(x^* = x_A^*\).
Indeed, suppose \(x^* = x_B^*\). Upon absorbing the \(O(\eta^2)x^*\) term,
the second inequality in \eqref{eq:partwise-perron-cap} yields $x^*\le(1+O(\eta^2)+o(1))\sqrt{a/(2m)}$.
Meanwhile, the first assertion of Lemma~\ref{lem:perron-mass}
guarantees that $x_u\geq(1-O(\eta^2))\sqrt{b/(2m)}$ for all $u\in A^*$.
Since $b\ge 100a$, $x_u>x^*$ for all $u\in A^*$, a contradiction.
Hence, $u^*\in A$, and the same lower bound for $u\in A^*$ together with the first inequality in \eqref{eq:partwise-perron-cap} gives
\begin{equation}\label{bbbbb}
 x_{u^*}=\big(1+O(\eta^2)+o(1)\big)\sqrt{\frac{b}{2m}}.
\end{equation}

Now, we show \(d_B(u^*) \ge (1-3\eta^2)\,b\).
Suppose \(d_B(u^*)<(1-3\eta^2)\,b\). Since $u^*\in A$,
Lemma \ref{lem:small} gives \(d_A(u^*) \le 2\eta^2 a\).
Together with \eqref{ooooo}, this yields
$$\rho^2x_{u^*}\leq \big(d_A(u^*)b+d_B(u^*)a+o(m)\big)x_{u^*}
\le \bigl((1-\eta^2)ab + o(m)\bigr)x_{u^*}.$$
Recall that $ab=\big(1+o(1)\big)m$ and $\eta^2$ is constant. Thus we have $\rho^2<m$,
which contradicts the lower bound in \eqref{eq:core-rhosq}.
Consequently, \(d_B(u^*) \ge (1-3\eta^2)\,b\).

Next, since $\sqrt{b/a}\ge 10$,
the second inequality in \eqref{eq:partwise-perron-cap} together with \eqref{bbbbb} directly implies
$x^*_B\le\frac14x_{u^*},$ i.e., $x_v\le\frac14x_{u^*}$ for any vertex $v\in B$.
Now take $u\in A$ with $x_u>\frac35x_{u^*}$, and set
$\beta_u=d_B(u)/b$.  Again by Lemma~\ref{lem:small}, we have
$d_A(u)\le2\eta^2a$.
Together with \eqref{ooooo} and $x_B^*\le \frac14x_{u^*}$, this yields
$$\rho^2x_u\leq d_A(u)bx_B^*+\big(d_B(u)a+o(m)\big)x_{u^*}
\le \bigl(ab(\frac12\eta^2+\beta_u)+o(m)\bigr)x_{u^*}.$$
Recall that $ab=\big(1+o(1)\big)m$ and $\rho^2>m$.
Thus, $x_u\leq \big(\frac12\eta^2+\beta_u+o(1)\big)x_{u^*}.$
Combining this with $x_u>\frac35x_{u^*}$, we obtain
$\beta_u>11/20$ for sufficiently small $\eta$ and large $m$.

Secondly, we treat the case $a=O_k(1)$.
Now, $b=\Theta(m)$ and $D=o(m)=o(ab)$.
Since $\eta^2$ is constant, and the total number of missing $A$--$B$ edges is at most $D$,
we have $$d_B(u)\geq b-D=\big(1-o(1)\big)b\geq\big(1-3\eta^2\big)b>11b/20$$ for all vertices \(u\in A\).
This is in fact stronger than the corresponding statements in (i) and (iii).

We next verify that $u^*\in A$ and $x_v=o(1)$ for every $v\in B$.
Let $K$ be the union of $K_{A,B}$ and $|R|$ isolated vertices.
Denote $M=A(G)-A(K)$, and let \(\rho(M)\) denote the maximum absolute value of the eigenvalues of $M.$
Each of the $D$ edge discrepancies between $G$ and $K$ contributes two nonzero off-diagonal entries,
each equal to $1$ or $-1$, to $M$. Hence,
$\operatorname{tr}(M^2)=2D,$
and thus $\rho(M)\leq\sqrt{2D}$.
Note that $\rho(K)=\sqrt{ab}$. By Weyl's inequality,
$|\rho(G)-\sqrt{ab}|
\leq \sqrt{2D}
=o(\sqrt {ab}),$
which yields
\begin{equation}\label{44}
\rho(G)\geq\big(1-o(1)\big)\sqrt{ab}.
\end{equation}

Let $p=\sum_{u\in A}x_u^2$ and $q=\sum_{v\in B}x_v^2$.
Then, $p+q=1$.
Since $\bm x^{\top}M \bm x\leq \rho(M)\leq \sqrt{2D}$, we have
$$\rho(G)=\bm x^{\top}A(K)\bm x+\bm x^{\top}M\bm x\leq
2\Big(\sum_{u\in A}x_u\Big)
\Big(\sum_{v\in B}x_v\Big)+\sqrt{2D}.$$
Applying the Cauchy--Schwarz inequality on $A$ and $B$ yields
$$\rho(G)\leq2\sqrt{ap}\sqrt{b\,q}+\sqrt{2D}=(2\sqrt{p\,q}+o(1))\sqrt{ab}.$$
Together with \eqref{44}, this gives
$1-o(1)\leq 2\sqrt{p\,q}$.
Combining this with $2\sqrt{p\,q}\leq p+q=1$ implies
$p,q=\frac12+o(1)$.
Together with $a=O_k(1)$, this implies
$$x^*_A\geq \sqrt{\frac1a\sum_{u\in A}x_u^2}=\sqrt{\frac pa}\geq c_k$$
for some positive constant $c_k$ that depends only on $k$.

Let $v\in B$.
Every neighbor of $v$ outside $A$ is incident to an edge in the edit-error set between $G$ and $K$.
Hence,
$|N(v)\setminus A|\leq D.$
By the eigen-equation and the Cauchy--Schwarz inequality,
\[
\begin{aligned}
\rho x_v
\leq
\sum_{u\in A}x_u
+\sum_{u\in N(v)\setminus A}x_u
\leq
\sqrt{ap}
+\sqrt{D}
\sqrt{\sum_{u\in N(v)\setminus A}\!\!\!x_u^2}
\leq \sqrt{ap}+\sqrt D.
\end{aligned}
\]
Recall that $a=O_k(1)$, $\sqrt{D}=o(\sqrt{ab})$ and $\rho>\sqrt{m}=\Theta(\sqrt {ab})$.
Thus $x_v=o(1)$.
Together with \(x_{u^*}\ge x_A^*\ge c_k\), this implies $u^*\in A$ and \(x_v=o(x_{u^*})\) for all \(v\in B\).
The proof is therefore complete.
\end{proof}

\section{Proof of Theorem \ref{thm1}}

Let $G$ be a $W_5$-free graph of sufficiently
large size $m$. Following the standing convention of this paper,
all isolated vertices have been removed. If \(\rho(G)<g^*(m)\),
where \(g^*(m)=\frac12\big(1+\sqrt{4m+1}\big)\) is defined in \eqref{eq:gasymp},
then Theorem \ref{thm1} holds trivially.
Hence, it remains only to consider
\begin{equation}
        \rho(G)\ge g^*(m).
        \label{eq:thm1-start}
\end{equation}

We will show that equality must hold in \eqref{eq:thm1-start} and that $G$ possesses precisely the structure described in the theorem.
This establishes the upper bound and the equality characterization simultaneously, thereby completing the proof of Theorem \ref{thm1}.
Applying Lemma \ref{lem:core-reduction} with $k=2$ and $\varepsilon=\varepsilon_0$,
we obtain a connected $\varepsilon_0$-core $H\subseteq G$.
Let $h=e(H)$ and $\rho=\rho(H).$ Lemma \ref{lem:core-reduction} gives
\begin{equation}
        h=(1-o(1))m
        \quad\text{and}\quad
        \rho\ge g^*(h)~~ \mbox{(i.e., $\rho^2-\rho\geq h$)}.
        \label{eq:thm1-core-threshold}
\end{equation}
Furthermore, by Lemma \ref{lem:core-reduction}, $\rho=g^*(h)$ implies $H=G$.

\begin{proof}[\bf Proof of Theorem \ref{thm1}]
We now apply the structural conclusions from Section \ref{sec4} to the core $H$.
Let $A,B$ be its minimum-edit partition chosen so that $|V(H)\setminus (A\cup B)|$ is also minimized. By Lemma \ref{lem:R-empty},
$V(H)=A\cup B.$
Write $a=|A|$ and $b=|B|,$
and swap the two parts if necessary to ensure \(a\le b\). There are precisely two cases.

\medskip
\noindent\textbf{Case 1.} $b\ge100a$.

Let $\bm x=(x_v)$ be the Perron vector of $H$ and choose $u^*$ with
$x_{u^*}=\max_{v\in V(H)}x_v$.  Lemma \ref{lem:split-perron-hierarchy} gives
$u^*\in A$, $d_B(u^*)\ge(1-3\eta^2)b$ and
$x_v\le \frac14x_{u^*}$ for all $v\in B$.
We now proceed with two claims.

%\begin{claim}\label{clm5.1}
%Except for at most one vertex $u_0$,
%every vertex \(u\in N_A(u^*)\) satisfies \(x_{v_1}+x_{v_2}\le x_{u^*}\) for all edges \(v_1v_2\in E(H[N(u^*)])\) not incident to $u$.
%\end{claim}

\begin{claim}\label{clm5.1}
There exists a set $X\subseteq N_A(u^*)$ with $|X|\le1$
such that every edge
$v_1v_2\in E(H[N(u^*)])$ not incident to any vertex of $X$
satisfies $x_{v_1}+x_{v_2}\le x_{u^*}$.
\end{claim}

\begin{proof}
For \(u\in N_A(u^*)\), let \(Y_u=N_B(u^*)\cap N_B(u)\).
In view of \eqref{eq:minedit}, every vertex in $A$ has at most \(b/2\) non-neighbors in $B.$
Thus, \(|Y_u|\ge d_B(u^*)-b/2 \ge \bigl(\tfrac12-3\eta^2\bigr)b\). If \(u,u'\) are distinct vertices in
\(N_A(u^*)\), then \(|Y_u\cap Y_{u'}|\le 1\); otherwise two distinct vertices
\(v,v'\in Y_u\cap Y_{u'}\) yield the 4-cycle
\(uvu'v'u\) within \(H[N(u^*)]\),
which creates a copy of \(W_5\) in $H.$
If \(u^*\) possessed three distinct neighbors \(u_1,u_2,u_3\in A\), the inclusion–exclusion principle would give
$$b=|B|\ge|Y_{u_1}\cup Y_{u_2}\cup Y_{u_3}|
\ge 3\bigl(\tfrac12-3\eta^2\bigr)b-3 > b
$$
for sufficiently large $h$, again a contradiction.
Hence, \(d_A(u^*)\le 2\).

If $d_A(u^*)=0$, take $X=\varnothing$.
Then every edge of $H[N(u^*)]$ lies inside \(B\), and hence
$x_{v_1}+x_{v_2}\le \frac12x_{u^*}$.
If $d_A(u^*)=1$, let $N_A(u^*)=\{u_0\}$ and take $X=\{u_0\}$.
Every edge not incident to \(u_0\) lies inside \(B\).
Hence, $x_{v_1}+x_{v_2}\le \tfrac12 x_{u^*}$, because $x_v\le \frac14x_{u^*}$ for all $v\in B$.

Now suppose that $d_A(u^*)=2$, so let
$N_A(u^*)=\{u_0,u_1\}$ with
$d_B(u_0)\ge d_B(u_1)$, and take $X=\{u_0\}$.
The bound $|Y_{u_0}\cap Y_{u_1}|\le1$ implies
$d_B(u^*)+d_B(u_0)+d_B(u_1)\le2b+1$, and hence
$d_B(u_1)\leq\frac12(2b+1-d_B(u^*))\le\frac12(b+1+3\eta^2b)<\frac{11}{20}b$ for sufficiently large $h$.
Now, the last assertion of Lemma \ref{lem:split-perron-hierarchy} gives
$x_{u_1}\le\frac35x_{u^*}$.
Thus, every edge of $H[N(u^*)]$ not incident
to $u_0$ either lies inside $B$ or joins $u_1$ to a vertex in $B$.
In both cases, the sum of the Perron coordinates at its endpoints is at most $(\frac14+\frac35)x_{u^*}$.
Therefore, the claim follows.
\end{proof}

\begin{claim}\label{clm5.2}
$\rho^2-\rho-h<0$.
\end{claim}

\begin{proof}
Write $U=N(u^*)$ and $W=V(H)\setminus(\{u^*\}\cup U)$.  Applying Lemma~\ref{lem:residual}
with $c=0$ and $d=1$, we obtain
\begin{equation}\label{resadiu}
 (\rho^2-\rho-h)x_{u^*}
 =\!\!\sum_{v_1v_2\in E(U)}\!\!(x_{v_1}+x_{v_2}-x_{u^*})
  -\rho x_{u^*}-\sum_{w\in W}f(w).
\end{equation}
Since $x_{u^*}$ is maximal, $f(w)\ge0$ for all $w\in W$.
Let $X$ be the set given by Claim \ref{clm5.1}.
If $X=\varnothing$, then every edge of $H[U]$
has nonpositive surplus, and hence the right-hand side is strictly negative.
Otherwise, since $|X|\le1$, we may write $X=\{u_0\}$.
By Claim \ref{clm5.1}, all positive surplus is carried by edges incident to $u_0$, and hence
\[
 \sum_{v_1v_2\in E(U)}\!\!(x_{v_1}+x_{v_2}-x_{u^*})
 \le\sum_{v\in N_U(u_0)}\!\!(x_{u_0}+x_v-x_{u^*})
 \le\sum_{v\in N_U(u_0)}\!\!x_v.
\]
As $u^*u_0\in E(H)$, the eigen-equation at $u_0$ gives
$\sum_{v\in N_U(u_0)}x_v\le(\rho-1)x_{u^*}$.  Substituting this into \eqref{resadiu} yields
$(\rho^2-\rho-h)x_{u^*}\le-x_{u^*}$, and the claim follows.
\end{proof}

Claim \ref{clm5.2} contradicts the second inequality in \eqref{eq:thm1-core-threshold}.
Hence, Case 1 cannot occur.
In the following, it suffices to consider the case $b<100a$.

\medskip
\noindent\textbf{Case 2.} $b<100a$.

In this case, Lemma \ref{cor:bounded-ratio-perron-cap} implies that
$d_B(u)>(1-\eta)b$ for every $u\in A$ and
$d_A(v)>(1-\eta)a$ for every $v\in B$.

\begin{claim}\label{clm5.3}
$\Delta(H[A])\le1$ and $\Delta(H[B])\le1$.
\end{claim}

\begin{proof}
Suppose to the contrary that some $u_1\in A$ has two distinct neighbors $u_2,u_3$ within $A$.
Then
\[
 \big|N_B(u_1)\cap N_B(u_2)\cap N_B(u_3)\big|\geq \sum_{i=1}^3d_B(u_i)-2\big|\cup_{i=1}^3N_B(u_i)\big|\geq(1-3\eta)b\ge2
\]
for sufficiently large $h$.
Taking two distinct vertices $v_1,v_2$ from this common neighborhood,
we see that $u_2v_1u_3v_2u_2$ is a $4$-cycle in $H[N(u_1)]$.
Together with the center $u_1$, this yields a forbidden $W_5$.
Consequently, $\Delta(H[A])\le 1$, and the same reasoning applies to $H[B]$.
\end{proof}

Based on Claim \ref{clm5.3},
Proposition \ref{lem:w5-matching} now gives $\rho^2-\rho\le h$.
Combining this with the second inequality in \eqref{eq:thm1-core-threshold}
forces equality \(\rho^2-\rho=h\), or equivalently, $\rho=g^*(h).$
By Lemma \ref{lem:core-reduction}, we now have $H=G$.
Again by Proposition \ref{lem:w5-matching}, we immediately conclude that
$G$ possesses precisely the structure described in Theorem \ref{thm1}.
This completes the proof.
\end{proof}

\section{Proof of Theorem \ref{thm2}}

Let $G$ be a $W_{2k+1}$-free graph of sufficiently large size $m$.
If \(\rho(G)<g_k(m)\),
where \(g_k(m)=\frac12\big(k-1+\sqrt{4m-k^2+1}\big)\) is defined in \eqref{eq:gdef},
then Theorem \ref{thm2} holds trivially.
Hence, it remains only to consider $\rho(G)\ge g_k(m)$.
It suffices to show that $G$
must in fact satisfy
$\rho(G)=g_k(m)$
and $G\cong K_k\vee qK_1$,
where $m=\binom{k}{2}+kq$.

\begin{proof}[{\bf Proof of Theorem \ref{thm2}}]
By Lemma \ref{lem:core-reduction},
we see that $G$ contains a connected $\varepsilon_0$-core, say $H$, such that
$h:=e(H)=(1-o(1))m$ and
$\rho:=\rho(H)\ge g_k(h)>\sqrt h$.
Let $S$ and $T$ be two disjoint vertex subsets of $H$. We denote by \(H[S,T]\) the bipartite subgraph of $H$
with vertex set \(S\cup T\) whose edge set consists of all edges of $H$ with one endpoint in $S$ and the other in $T$.
We write $E(S)=E(H[S])$ and $E(S,T)=E(H[S,T])$ for simplicity.
Moreover, define
$e(S)=e(H[S])$ and $e(S,T)=e(H[S,T]).$

In what follows, we apply all results of Section~\ref{sec4} to $H$,
with $h=e(H)$ in place of the parameter $m$ used there.
Let $A,B$ be the minimum-edit partition from
Section~\ref{sec4}.  Lemma~\ref{lem:R-empty} gives $V(H)=A\cup B$, and we relabel
the parts so that $a=|A|\le b=|B|$.
Let $\bm{x}=(x_v)$ be the Perron vector of $H$.
Choose $u^*\in V(H)$ with
$x_{u^*}=\max_{v\in V(H)}x_v.$
Following the proof framework of Theorem \ref{thm1},
we consider two cases: the first yields \(H=G\cong K_k\vee qK_1\),
while the second leads to a contradiction.

\medskip
\noindent\textbf{Case 1.} $b\ge100a$.

By Lemma \ref{lem:split-perron-hierarchy}, we have \(u^*\in A\),
\(d_B(u^*)\ge(1-3\eta^2)b\),
every $v\in B$ has \(x_v\leq\frac14x_{u^*}\),
and each \(u\in A\) with \(x_u>\frac{3}{5}x_{u^*}\) satisfies
\(d_B(u)>\frac{11}{20}b\).
Write
$L=\big\{u\in N_A(u^*): x_u>\frac35x_{u^*}\big\}$ and $L^*=N(u^*)\setminus L$.
Furthermore, we define
\[S_{L^*}
=\!\!\sum_{yz\in E(L^*)}\!\!(x_y+x_z-x_{u^*}) \quad \mbox{and} \quad
S_L=\!\!\sum_{yz\in E(N(u^*))\setminus E(L^*)}
\!\!\!(x_y+x_z-x_{u^*}).
\]

\begin{claim}\label{claim6.1}
We have $|L|\le k-1$.
\end{claim}

\begin{proof}
Suppose for contradiction that $|L|\ge k$, and choose distinct vertices
$v_1,\ldots,v_k\in L$, with indices taken cyclically.
Since $v_i\in L$, Lemma~\ref{lem:split-perron-hierarchy} gives
$d_B(v_i)>\frac{11}{20}b$,
while $d_B(u^*)\ge(1-3\eta^2)b$.
Hence, for every $i\in\{1,\ldots,k\}$, we have
\begin{align*}
\big|N_B(u^*)\cap N_B(v_i)\cap N_B(v_{i+1})\big|
\ge
d_B(u^*)+d_B(v_i)+d_B(v_{i+1})-2b
>\big(\frac1{10}-3\eta^2\big)b.
\end{align*}
For sufficiently small $\eta$ and sufficiently large $h$, the
right-hand side is larger than $k$.  We may therefore choose pairwise
distinct vertices
$w_i\in
N_B(u^*)\cap N_B(v_i)\cap N_B(v_{i+1})$
for each $i\in\{1,\ldots,k\}$.
Thus,
$v_1w_1v_2w_2\cdots v_kw_kv_1$
is a copy of $C_{2k}$ in $H[N(u^*)]$.  Together with the center $u^*$,
this gives a copy of $W_{2k+1}$, a contradiction. Hence,
$|L|\le k-1$.
\end{proof}

\begin{claim}\label{claim6.2}
$S_{L^*}\le0$, with equality only if $L^*\cap A=\varnothing$ and $e(L^*)=0$.
\end{claim}

\begin{proof}
By the definition of $L$,
$x_u\le\frac35x_{u^*}$ for every $u\in L^*\cap A$,
while Lemma \ref{lem:split-perron-hierarchy} gives
$x_v\le\frac14x_{u^*}$ for every $v\in L^*\cap B$.
Consequently, an edge inside $L^*\cap A$ contributes at most
$\frac 15 x_{u^*}$ to $S_{L^*}$, an edge between $L^*\cap A$ and
$L^*\cap B$ contributes at most $-\frac{3}{20}x_{u^*}$, and an edge inside
$L^*\cap B$ contributes at most $-\frac 12 x_{u^*}$.  Hence
\begin{align}\label{eq:39A}
S_{L^*}
\leq
\frac15 e\big(L^*\cap A\big)x_{u^*}
-\frac3{20}e\big(L^*\cap A,L^*\cap B\big)x_{u^*}
-\frac12e\big(L^*\cap B\big)x_{u^*}.
\end{align}

Suppose now that $L^*\cap A\ne\varnothing$.
In view of \eqref{eq:minedit}, $d_B(u)\ge\frac b2$ for every $u\in L^*\cap A$.
Since also $d_B(u^*)\ge(1-3\eta^2)b$, we have
$|N_B(u)\cap N_B(u^*)|\geq(\frac12-3\eta^2)b$.
Since $L\subseteq A$, removing $L$ from $N(u^*)$ removes no vertex of
$B$, and therefore
$L^*\cap B=N(u^*)\cap B$.
Consequently,
$d_{L^*\cap B}(u)\geq (\frac12-3\eta^2)b$
for every $u\in L^*\cap A$.
Summing over all $u\in L^*\cap A$, we obtain
\begin{align*}
e\big(L^*\cap A,L^*\cap B\big)
=\sum_{u\in L^*\cap A}\!\!d_{L^*\cap B}(u)
\ge\big(\frac12-3\eta^2\big)b\,\big|L^*\cap A\big|.
\end{align*}

On the other hand,
$e(L^*\cap A)
\le \frac a2|L^*\cap A|$.
Substituting these estimates into \eqref{eq:39A} and
discarding the last nonpositive term gives
\[
S_{L^*}
\le
\Big(
\frac a{10}
-\frac3{20}
\big(\frac12-3\eta^2\big)b
\Big)
\Big|L^*\cap A\Big|x_{u^*}.
\]
Since $b\ge100a$ and $\eta$ is sufficiently small,
we have $S_{L^*}<0$ whenever $L^*\cap A\ne\varnothing$.
Finally, if $L^*\cap A=\varnothing$, then $L^*\subseteq B$, and every
$yz\in E(L^*)$ satisfies
$x_y+x_z-x_{u^*}\le-\frac12x_{u^*}<0.$
Hence, $S_{L^*}\le0$, with equality only when $L^*\cap A=\varnothing$ and
$e(L^*)=0$.
\end{proof}

Write $P=\rho^2-(k-1)\rho-h+\binom{k}{2}$.
Since $\rho\geq g_k(h)$, it follows that $P\geq 0$.
Define $U=N(u^*)$ and $W=V(H)\setminus(\{u^*\}\cup N(u^*))$.
Applying Lemma \ref{lem:residual} with $d=k-1$ and $c=\binom{k}{2}$, we obtain
\begin{align}\label{eq:41A}
Px_{u^*}
=\!\!\sum_{yz\in E(U)}\!\!\!(x_y+x_z-x_{u^*})-(k-1)\rho x_{u^*}+\binom{k}{2}x_{u^*}-\sum_{w\in W}f(w).
\end{align}
Since $U=N(u^*)=L\cup L^*$, the edges in $E(U)$ can be divided into
the edges contained entirely in $L^*$ and the edges having at least
one endpoint in $L$.
Then $\sum_{yz\in E(U)}(x_y+x_z-x_{u^*})=S_L+S_{L^*}$.

We now compute $S_L$ by summing the eigen-equations over the
vertices of $L$.
For every $v\in L$, the eigen-equation at $v$ is
\[
\rho x_v
=
x_{u^*}
+\!\!\sum_{z\in N_L(v)}\!\!x_z
+\!\!\sum_{z\in N_{L^*}(v)}\!\!x_z
+\!\!\sum_{z\in N_{W}(v)}\!\!x_z.
\]
Summing this identity over all $v\in L$ gives
\begin{align}\label{eq:42A}
\rho\sum_{v\in L}x_v
={}&
|L| x_{u^*}
+\!\!\sum_{yz\in E(L)}\!\!(x_y+x_z)
+\sum_{v\in L}\sum_{z\in N_{L^*}(v)}x_z
+\sum_{w\in W}d_L(w)x_w.
\end{align}
Indeed, every edge $yz\in E(L)$ is counted twice in the
summation over $L$, once as $x_z$ in the equation at $y$ and once
as $x_y$ in the equation at $z$, and hence contributes $x_y+x_z$.

On the other hand, the definition of $S_L$ yields
\begin{align*}
S_L
&=
\sum_{yz\in E(L)}\!\!(x_y+x_z-x_{u^*})
+\!\!\sum_{vz\in E(L,L^*)}\!\!(x_v+x_z-x_{u^*})\\
&=\sum_{yz\in E(L)}\!\!(x_y+x_z)
+\!\!\sum_{v\in L}\,\sum_{z\in N_{L^*}(v)}\!\!x_z
+\sum_{v\in L}d_{L^*}(v)x_v
-\bigl(e(L)+e(L,L^*)\bigr)x_{u^*}\\
&=\rho\sum_{v\in L}x_v
-|L| x_{u^*}
-\sum_{w\in W}\!\!d_L(w)x_w
+\sum_{v\in L}d_{L^*}(v)x_v
-\bigl(e(L)+e(L,L^*)\bigr)x_{u^*},
\end{align*}
where the last equality follows from \eqref{eq:42A}.
Since
$e(L,L^*)=\sum_{v\in L}d_{L^*}(v)$,
we have
\[
\sum_{v\in L}d_{L^*}(v)x_v
-e(L,L^*)x_{u^*}
=
-\sum_{v\in L}
d_{L^*}(v)(x_{u^*}-x_v).
\]
Also, we have
$\sum_{v\in L}x_v
=|L| x_{u^*}-\sum_{v\in L}(x_{u^*}-x_v)$.
Consequently,
\begin{align*}%\label{eq:43A}
S_L=
\big(|L|\rho-|L|-e(L)\big)x_{u^*}
-\sum_{v\in L}\big(\rho+d_{L^*}(v)\big)(x_{u^*}-x_v)
-\!\!\sum_{w\in W}\!\!d_L(w)x_w.
\end{align*}
Recall that \(\sum_{yz\in E(U)}(x_y+x_z-x_{u^*})=S_L+S_{L^*}\).
Substituting the equality above into \eqref{eq:41A} gives the following residual identity,
which will be used in the subsequent arguments.
\begin{align}\label{eq:40A}
Px_{u^*}
={}&
S_{L^*}
-\big(k-1-|L|\big)\rho x_{u^*}
+\Big(\binom{k}{2}-|L|-e(L)\Big)x_{u^*}
\nonumber\\
&-\sum_{v\in L}\big(\rho+d_{L^*}(v)\big)(x_{u^*}-x_v)
-\sum_{w\in W}\big(d_L(w)x_w+f(w)\big).
\end{align}

\begin{claim}\label{claim6.3}
We have $|L|=k-1$.
\end{claim}

\begin{proof}
We first analyze the signs of the terms in
\eqref{eq:40A}.
Since $x_{u^*}$ is the largest Perron coordinate, $x_{u^*}\geq x_v>0$ for all $v\in V(H)$,
and thus $f(w)\geq0$ for each $w\in W$.
Moreover, Claim \ref{claim6.2} gives
$\sum_{yz\in E(L^*)}(x_y+x_z-x_{u^*})\le0.$
Consequently,
discarding all the nonpositive terms in
\eqref{eq:40A}, and using
$\binom{k}{2}-|L|-e(L)\le \binom{k}{2}$,
we obtain
\begin{align*}
Px_{u^*}\le-\big(k-1-|L|\big)\rho x_{u^*}+\binom{k}{2}x_{u^*}.
\end{align*}

By Claim \ref{claim6.1}, we have $|L|\leq k-1$.
Now suppose that $|L|\le k-2$.
Then $Px_{u^*}\le\big(-\rho+\binom{k}{2}\big)x_{u^*}.$
For sufficiently large $h$, we have $\rho\geq g_k(h)>\binom{k}{2}$, and therefore
$P<0$. This leads to a contradiction.
Therefore, $|L|=k-1$.
\end{proof}

Set $\sigma:=\binom{|L|}2-e(L)$.
Since $|L|=k-1$ by Claim \ref{claim6.3}, and $\binom{k}2-(k-1)=\binom{|L|}2$,
identity \eqref{eq:40A} now simplifies to
\begin{align}\label{eq:44A}
Px_{u^*}=
S_{L^*}+\sigma x_{u^*}
-\sum_{v\in L}\big(\rho+d_{L^*}(v)\big)(x_{u^*}-x_v)
-\sum_{w\in W}\!\big(d_L(w)x_w+f(w)\big).
\end{align}

\begin{claim}\label{claim6.2G}
If $\sigma>0$, then we must have
$\sum_{v\in L}\big(\rho+d_{L^*}(v)\big)(x_{u^*}-x_v)
+\sum_{w\in W}d_L(w)x_w
>\sigma x_{u^*}.$
\end{claim}

\begin{proof}
Fix $v\in L\subseteq N(u^*)$.
Comparing the eigen-equations at $u^*$ and $v$, we obtain
$\rho(x_{u^*}-x_v)=(x_v-x_{u^*})+
\sum_{u\in N(u^*)\setminus N[v]}x_u-
\sum_{w\in N(v)\setminus N[u^*]}x_w,$
and hence
\begin{equation}\label{eq:45A}
\big(\rho+1\big)\big(x_{u^*}-x_v\big)
+\!\!\sum_{w\in N(v)\setminus N[u^*]}\!\!x_w
=\!\!\sum_{u\in N(u^*)\setminus N[v]}\!\!x_u.
\end{equation}

Since $W=V(H)\setminus N[u^*]$ and $v\in L\subseteq N(u^*)$,
we have $N(v)\setminus N[u^*]=N_W(v)$.
Furthermore, interchanging the order of summation gives
$\sum_{v\in L}\sum_{w\in N_W(v)}x_w=\sum_{w\in W}d_L(w)x_w.$
Therefore, summing equality \eqref{eq:45A} over all vertices in $L$ yields
\begin{equation}\label{eq:46A}
\big(\rho+1\big)\sum_{v\in L}\big(x_{u^*}-x_v\big)
+\sum_{w\in W}d_L(w)x_w
=\sum_{v\in L}\,
\sum_{u\in N(u^*)\setminus N[v]}\!\!x_u.
\end{equation}

We next analyze the right-hand side of \eqref{eq:46A}.
Recall that $N(u^*)=L\cup L^*$. For each $v\in L$, any vertex in
$N(u^*)\setminus N[v]$ is a non-neighbor of $v$, lying either in $L$ or in $L^*$.  Thus,
$\sum_{v\in L}\sum_{u\in N(u^*)\setminus N[v]}x_u
\geq\sum_{v\in L}\sum_{\substack{u\in L\setminus N[v]}}x_u,$
where the right-hand term can be
rewritten in terms of the non-edges of $E(L)$.  Indeed, if
$uv$ is a non-edge of $E(L)$, then $u,v\in N(u^*)$ while
$uv\notin E(H)$. Hence, $v\in N(u^*)\setminus N[u]$, and so the
summand corresponding to $u$ contributes $x_v$.  Similarly,
$u\in N(u^*)\setminus N[v]$, and so the summand corresponding to
$v$ contributes $x_u$.  Consequently, every non-edge
$uv$ of $E(L)$ contributes exactly $x_u+x_v$, and therefore
\[
\sum_{v\in L}\,
\sum_{u\in N(u^*)\setminus N[v]}\!\!x_u\geq\sum_{v\in L}\,
\sum_{\substack{u\in L\setminus N[v]}}\!\!x_u
=\!\!\sum_{uv\in\binom{L}{2}\setminus E(L)}
\!\!\!\big(x_u+x_v\big).
\]
Substituting this inequality into \eqref{eq:46A} yields
\begin{equation}\label{eq:49A}
\big(\rho+1\big)\sum_{v\in L}\big(x_{u^*}-x_v\big)
+\sum_{w\in W}d_L(w)x_w
\ge\sum_{uv\in\binom{L}{2}\setminus E(L)}\!\!
\big(x_u+x_v\big).
\end{equation}

Now suppose that $\sigma=\binom{|L|}{2}-e(L)>0$.
The graph $H[L]$ has exactly $\sigma$ missing edges.
Moreover, by the definition of $L$,
$x_v>\frac35x_{u^*}$ for every $v\in L$.
Thus, for every non-edge $uv$ of $E(L)$,
we have $x_u+x_v>\frac65x_{u^*}$.
It follows from \eqref{eq:49A} that
\begin{equation}\label{eq:50A}
(\rho+1)\sum_{v\in L}\big(x_{u^*}-x_v\big)
+\sum_{w\in W}d_L(w)x_w
>\frac65 \sigma x_{u^*}.
\end{equation}

We also need to control the term involving $d_{L^*}(v)$.
For every $v\in L$,
by conclusion (iii) of Lemma \ref{lem:split-perron-hierarchy},
we have $d_B(v)>\frac{11}{20}b$.
Thus,
\[
\big|N_B(u^*)\cap N_B(v)\big|
>\big(\frac{11}{20}-3\eta^2\big)b>0.
\]
Since $L\subseteq A$, we have
$N_B(u^*)\subseteq (N(u^*)\setminus L)=L^*$.
Hence,
$N_B(u^*)\cap N_B(v)\subseteq N_{L^*}(v)$.
Therefore, $d_{L^*}(v)\ge1$.
Since $x_{u^*}-x_v\ge0$, it follows that
$(\rho+d_{L^*}(v))(x_{u^*}-x_v)
\ge(\rho+1)(x_{u^*}-x_v)$ for all $v\in L$.
Combining this with \eqref{eq:50A}, we obtain
\begin{align*}
\sum_{v\in L}\!\big(\rho+d_{L^*}(v)\big)\big(x_{u^*}-x_v\big)
+\sum_{w\in W}d_L(w)x_w
>\frac65 \sigma x_{u^*}
>\sigma x_{u^*},
\end{align*}
as required. This proves the claim.
\end{proof}

\begin{claim}\label{claim6G}
$H[L]\cong K_{k-1}$ and $\sigma=0$.
\end{claim}

\begin{proof}
Recall from Claim \ref{claim6.2} that $S_{L^*}\leq0$.
Suppose now that $H[L]$ is not a complete graph.
Then $\sigma=\binom{|L|}2-e(L)>0$.
Returning to \eqref{eq:44A}, the only positive term on the right-hand side of
\eqref{eq:44A} is $\sigma x_{u^*}$.
Combining this with Claim \ref{claim6.2G} yields
$Px_{u^*}<0$,
and thus $P<0$, which contradicts the assumption $\rho\geq g_k(h)$.
Hence, $H[L]\cong K_{k-1}$ and $\sigma=0$.
\end{proof}

Now by Claim \ref{claim6G},
$\sigma=0$,
and equality \eqref{eq:44A} further reduces to
\begin{align}\label{eq:51A}
Px_{u^*}
=S_{L^*}-\sum_{v\in L}\big(\rho+d_{L^*}(v)\big)\big(x_{u^*}-x_v\big)
-\sum_{w\in W}\big(d_L(w)x_w+f(w)\big).
\end{align}

Every term on the right-hand side of
\eqref{eq:51A}, regarded as one of the displayed
aggregate terms, is nonpositive.  Hence
$P\le0$.
Together with $P\ge0$, this gives $P=0$.
Since $P=0$ and $\rho\ge g_k(h)$, we have $\rho(H)=g_k(h)$.
Lemma \ref{lem:core-reduction} now yields $H=G$ and $h=m$.

Since every term on the right-hand side of \eqref{eq:51A} is nonpositive,
each of the three terms therein must equal zero.
Firstly,
$\sum_{v\in L}(\rho+d_{L^*}(v))(x_{u^*}-x_v)=0$.
Thus,
$x_v=x_{u^*}$ for every $v\in L$.
Secondly,
$\sum_{w\in W}(d_L(w)x_w+f(w))=0$.
Since all Perron coordinates are positive and
$f(w)\ge0$, it follows that
$d_L(w)=f(w)=0$ for all $w\in W.$
Finally,
$S_{L^*}=0$.
By the equality condition in Claim \ref{claim6.2}, we know that
$L^*\cap A=\varnothing$ and $e(L^*)=0$.

We first prove that $W=\varnothing$.
Suppose for contradiction that there exists some $w\in W$. Recall that
$f(w)=d_{N(u^*)}(w)(x_{u^*}-x_w)+\frac12d_W(w)x_{u^*}.$
Since $x_{u^*}\ge x_w>0$, both terms on the right-hand side are nonnegative.
As \(f(w)=0\), we deduce that $d_W(w)=0$ and \(d_{N(u^*)}(w)(x_{u^*}-x_w)=0\).
Since $H$ is connected and \(w\notin N[u^*]\),
the condition $d_W(w)=0$ ensures that $w$ has at least one neighbor in $N(u^*)$.
Consequently,
$d_{N(u^*)}(w)>0$, which yields $x_w=x_{u^*}$.
Moreover, since $d_W(w)=0$, we have $N(w)\subseteq N(u^*)$.
By the eigenequations at \(w\) and \(u^*\),
$\sum_{z\in N(w)}x_z=\rho x_w=\rho x_{u^*}=\sum_{z\in N(u^*)}x_z$.
Since all Perron coordinates are positive and
\(N(w)\subseteq N(u^*)\), it follows that
$N(w)=N(u^*)=L\cup L^*$.
Thus $d_L(w)=|L|=k-1>0$,
contradicting $d_L(w)=0$.
Therefore, $W=\varnothing$, and so $V(H)=N[u^*]$.

Fix an arbitrary $v\in L$. Then $x_v=x_{u^*}$.
Since $W=\varnothing$, we have
$V(H)=N[u^*]$.
In particular,
$N[v]\subseteq N[u^*]$.
On the other hand, by the eigenequations at $v$ and $u^*$,
\[
\sum_{z\in N[v]}x_z
=(\rho+1)x_v
=(\rho+1)x_{u^*}
=\sum_{z\in N[u^*]}x_z.
\]
Since all Perron coordinates are positive, it follows from
$N[v]\subseteq N[u^*]$ that
$N[v]=N[u^*]=L\cup L^*\cup\{u^*\}$.
As $v\in L$ was arbitrary, $H[L,L^*]$ is complete bipartite.
Together with $H[L]\cong K_{k-1}$ and $e(L^*)=0$, this yields
$H=G\cong K_k\vee |L^*|K_1$.
Let $q=|L^*|$. Then
$m=\binom{k}{2}+kq$.

\medskip
\noindent\textbf{Case 2.} $b<100a$.

By Lemma \ref{cor:bounded-ratio-perron-cap}, we have the following cross-degree estimates:
\begin{equation}\label{eq:52A}
        d_B(u)>(1-\eta)b ~~ (u\in A)
        \quad \mbox{and}\quad
        d_A(v)>(1-\eta)a ~~ (v\in B).
\end{equation}

Since $ab=(1+o(1))h$ and $1\le\frac ba<100$,
we have $a,b=\Theta(\sqrt h)$.
In particular,
$a,b\to\infty$.
We now derive a contradiction through a sequence of claims.

\begin{claim}\label{claim:general-internal-degree}
We have $\Delta(H[A])\le k-1$ and $\Delta(H[B])\le k-1$.
\end{claim}

\begin{proof}
Indeed, suppose first that
$\Delta(H[A])\ge k$.
Then there exists a vertex $u\in A$ having $k$ distinct neighbors
$u_1,u_2,\ldots,u_k\in A$.
Read the indices cyclically, so that $u_{k+1}=u_1$.
For every $i\in\{1,\ldots,k\}$, it follows from \eqref{eq:52A} that
$$\big|N_B(u)\cap N_B(u_i)\cap N_B(u_{i+1})\big|
\ge d_B(u)+d_B(u_i)+d_B(u_{i+1})-2b
>(1-3\eta)b.$$
Since $k$ is fixed and $b\to\infty$, each of these $k$ triple
intersections has cardinality tending to infinity. Thus, for all
sufficiently large $h$, we may choose distinct vertices
$z_i\in N_B(u)\cap N_B(u_i)\cap N_B(u_{i+1})$
for each $i\in\{1,\ldots,k\}$.
Then, $u_1z_1u_2z_2\cdots u_kz_ku_1$ forms
a $2k$-cycle in $H[N(u)]$.
Hence, $G$ contains a copy of $W_{2k+1}$, contradicting the assumption that $G$ is
$W_{2k+1}$-free.
Therefore,
$\Delta(H[A])\le k-1$.
By symmetry, we can prove $\Delta(H[B])\le k-1$.
\end{proof}

Let $s=e(A)+e(B)$.
Define $\alpha=2e(A)/a$ and $\beta=2e(B)/b$.
From Claim \ref{claim:general-internal-degree},
we know that $\alpha,\beta\le k-1$.
Recall that $V(H)=A\cup B$. Thus
$e(A,B)=h-s$.

\begin{claim}\label{claim:general-F-bound}
Let $r=\sqrt{b/a}$
and define $\Gamma(r,\alpha,\beta):=\alpha(\frac12-\frac1{4r})
+\beta(\frac12-\frac r4)$.
Then
\begin{equation}\label{eq:53A}
        \rho
        \le
        \sqrt h+\Gamma(r,\alpha,\beta)
        +O_{k}(\eta^2)+o(1).
\end{equation}
\end{claim}

\begin{proof}
The Rayleigh quotient gives
\[
 \rho=\sum_{uv\in E(A,B)}\!\!2x_ux_v
 +\sum_{uv\in E(A)}\!\!2x_ux_v
 +\sum_{uv\in E(B)}\!\!2x_ux_v.
\]
For the cross contribution, $e(A,B)=h-s$.
By the Cauchy--Schwarz inequality,
\[
\begin{aligned}
 \big(\sum_{uv\in E(A,B)}\!\!\!x_ux_v\big)^2
\le\big(h-s\big)\!\!\!\sum_{uv\in E(A,B)}\!\!\!x_u^2x_v^2
\le \big(h-s\big)\sum_{u\in A}x_u^2\sum_{v\in B}x_v^2.
\end{aligned}
\]
Since $\sum_{u\in A}x_u^2+\sum_{v\in B}x_v^2=1$,
we have $\sum_{u\in A}x_u^2\sum_{v\in B}x_v^2\le\frac14$,
and therefore,
\begin{equation}\label{eq:54A}
 \sum_{uv\in E(A,B)}\!\!\!2x_ux_v
 \le\sqrt{h-s}.
\end{equation}
In view of \eqref{eq:balanced-cap}, we deduce that
\[
\begin{aligned}
 \sum_{uv\in E(A)}\!\!2x_ux_v
 \le 2e(A)(x_A^*)^2
 \le e(A)\frac{b}{h}
       \bigl(1+O(\eta^2)+o(1)\bigr)
 =\frac{e(A)}{a}+O_{k}(\eta^2)+o(1),
\end{aligned}
\]
where we used $ab=(1+o(1))h$ and
$e(A)/a=\alpha/2\le(k-1)/2$. Symmetrically,
\[
 \sum_{uv\in E(B)}\!\!2x_ux_v
 \le\frac{e(B)}{b}+O_{k}(\eta^2)+o(1).
\]
Combining \eqref{eq:54A} and the above two inequalities, we obtain
\begin{equation}\label{eq:55A}
 \rho=\sum_{uv\in E(H)}2x_ux_v
 \le
 \sqrt{h-s}
 +\frac{e(A)}{a}
 +\frac{e(B)}{b}
 +O_{k}(\eta^2)+o(1).
\end{equation}

Recall that $a\leq b<100a$.
Thus $1\leq r=\sqrt{b/a}<10$.
Furthermore, $ab=(1+o(1))h$ implies
$a=(1+o(1))\sqrt h/r$ and
$b=(1+o(1))r\sqrt h$.
Since $e(A)=\alpha a/2$, $e(B)=\beta b/2$ and $\alpha,\beta\leq k-1$,
it follows that
$s=(\alpha a+\beta b)/2=O_{k}(\sqrt h)$.
Therefore,
\[
\begin{aligned}
 \sqrt{h-s}=\sqrt h-\frac{s}{2\sqrt h}+o(1)
 =\sqrt h-\frac{\alpha}{4r}-\frac{\beta r}{4}+o(1).
\end{aligned}
\]
Moreover, $e(A)/a=\alpha/2$ and $e(B)/b=\beta/2$.
Substituting these estimates into
\eqref{eq:55A} yields
\[
\begin{aligned}
 \rho
 \le
 \sqrt h
 -\frac{\alpha}{4r}
 -\frac{\beta r}{4}
 +\frac{\alpha}{2}
 +\frac{\beta}{2}
 +O_{k}(\eta^2)+o(1)=\sqrt h
 +\Gamma(r,\alpha,\beta)
 +O_{k}(\eta^2)+o(1),
\end{aligned}
\]
as claimed.
\end{proof}

\begin{claim}\label{claim:general-alpha-beta-near-k}
We have $\min\{\alpha,\beta\}>k-1-1/4.$
\end{claim}

\begin{proof}
By Claim \ref{claim:general-internal-degree}, we have
$\max\{\alpha,\beta\}\leq k-1.$
Suppose to the contrary that
$\min\{\alpha,\beta\}\leq k-1-\frac14.$
Recall $1\leq r=\sqrt{b/a}<10$ and
$\Gamma(r,\alpha,\beta)=\alpha(\frac12-\frac1{4r})
+\beta(\frac12-\frac r4).$
We divide the proof according to the value of $r$.

Firstly, assume that $1\le r\le 2$.
In this case, both coefficients
$\frac12-\frac1{4r}$ and
$\frac12-\frac r4$ are nonnegative. Since
$\max\{\alpha,\beta\}\leq k-1$ and
$\min\{\alpha,\beta\}\leq k-1-\frac14$,
we obtain
\[
\begin{aligned}
\Gamma(r,\alpha,\beta)
\leq
\big(k\!-\!1\big)\big(\frac12\!-\!\frac1{4r}\big)
\!+\!\big(k\!-\!1\!-\!\frac14\big)
\big(\frac12\!-\!\frac r4\big)=\big(k\!-\!1\big)\big(1\!-\!\frac14(r\!+\!r^{-1})\big)
\!-\!\frac14\big(\frac12\!-\!\frac r4\big).
\end{aligned}
\]
If $1\leq r\leq\frac32$,
then $r+r^{-1}\geq2$ and $\frac12-\frac r4\geq\frac18$.
It follows that
$\Gamma(r,\alpha,\beta)
\leq\frac{k-1}{2}-\frac1{32}.$
If $\frac32<r\le2,$
then $r+r^{-1}\geq \frac32+\frac23=\frac{13}{6}$
and
$\Gamma(r,\alpha,\beta)
\leq\frac{11(k-1)}{24}
=\frac{k-1}{2}-\frac{k-1}{24}.$

It remains to consider the case $2<r<10$.
In this case, $\beta(\frac12-\frac r4)\leq0.$
Therefore, we obtain
$$\Gamma(r,\alpha,\beta)\leq\big(k-1\big)\big(\frac12-\frac1{4r}\big)\leq\frac{k-1}{2}-\frac{k-1}{40}.$$

Thus, in any case there exists
$\gamma=\min\{\frac1{32},\frac{k-1}{40}\}>0$
such that $\Gamma(r,\alpha,\beta)\le\frac{k-1}{2}-\gamma$.
Choose $\eta$ sufficiently small so that the
$O_{k}(\eta^2)$ term in \eqref{eq:53A} has absolute
value at most $\gamma/6$, and then choose $h$ sufficiently large
so that the $o(1)$ term has absolute value at most $\gamma/6$.
It follows from \eqref{eq:53A} that
\[
\begin{aligned}
 \rho
 \le
 \sqrt h+\frac{k-1}{2}
 -\gamma+\frac{\gamma}{6}+\frac{\gamma}{6}
 =\sqrt h+\frac{k-1}{2}-\frac{2\gamma}{3}.
\end{aligned}
\]
On the other hand,
$g_k(h) =\sqrt h+\frac{k-1}{2}+O_k(h^{-1/2})$,
and hence
$\rho\ge g_k(h)>\sqrt h+\frac{k-1}{2}-\frac{\gamma}{3}$,
which leads to a contradiction.
Therefore, $\min\{\alpha,\beta\}>k-1-\frac14.$
\end{proof}

By Claim \ref{claim:general-internal-degree},
$d_{A}(u)\le k-1$ for every $u\in A$.
We now prove that there exists some vertex $u_0\in A$ satisfying $d_A(u_0)=k-1$.
Indeed, if no vertex in $A$ has internal degree \(k-1\),
then \(d_{A}(u)\le k-2\) for every \(u\in A\), and thus
$\alpha
=\frac{2e(A)}{a}
=\frac1a\sum_{u\in A}d_{A}(u)
\le k-2,$
which contradicts the inequality \(\alpha>k-1-\tfrac14\) from Claim \ref{claim:general-alpha-beta-near-k}.

Let $N_A(u_0)=\{u_1,\ldots,u_{k-1}\}$.
In view of \eqref{eq:52A},
$|B\setminus N_B(u_i)|<\eta b$
for $i\in\{0,\ldots,k-1\}$.
Define
$X:=\bigcap_{i=0}^{k-1}N_B(u_i)$.
Since any \(v\in B\setminus X\) lies in \(B\setminus N_B(u_i)\) for some index $i,$
we have $B\setminus X\subseteq \bigcup_{i=0}^{k-1}(B\setminus N_B(u_i))$.
Thus $|B\setminus X|
\le\sum_{i=0}^{k-1}|B\setminus N_B(u_i)|
< k\eta b,$
and so $|X|=b-|B\setminus X|>(1-k\eta)b$.

Let $E_{\mathrm{out}}$ be the set of edges of $H[B]$ having at
least one endpoint within $B\setminus X$. By
Claim \ref{claim:general-internal-degree}, we have
$\Delta(H[B])\le k-1$.
Therefore,
\[
\begin{aligned}
 \big|E_{\mathrm{out}}\big|
\le\sum_{v\in B\setminus X}\!\!d_{B}(v)
\le \big(k\!-\!1\big)\big|B\setminus X\big|
<k\big(k\!-\!1\big)\eta b.
\end{aligned}
\]
By Claim \ref{claim:general-alpha-beta-near-k},
$\beta>k-1-\frac14$.
Thus
$e(B)=\frac{\beta b}{2}>(\frac{k\!-\!1}2-\frac18)\,b.$
Recall that $\eta\ll1$ and $b$ is sufficiently large.
It follows that $$ e(X)
= e(B)-|E_{\mathrm{out}}|
>\big(\frac{k\!-\!1}2-\frac18-k(k\!-\!1)\eta\big)b>2(k\!-\!1).$$

We next show that $H[X]$ contains at least two disjoint edges.
Indeed, suppose otherwise, and let $pq\in E(X)$.
Since there are no two disjoint
edges in $H[X]$, every edge of $H[X]$ must be incident with at least
one of $p$ and $q$. Therefore,
$
\begin{aligned}
 e(X)
\le d_{X}(p)+d_{X}(q)
\le 2(k-1),
\end{aligned}
$
a contradiction. Hence, $H[X]$ contains two disjoint edges, say
$p_1q_1$ and $p_2q_2$.

If $k=3$, then
$u_1p_1q_1u_2p_2q_2u_1$
is a 6-cycle.
Suppose next that $k\ge4$. Since
$|X|>(1-k\eta)b$, we may choose distinct vertices
$z_2,\ldots,z_{k-2} \in X\setminus\{p_1,q_1,p_2,q_2\}$.
Then
\[
        u_1p_1q_1u_2z_2u_3z_3
        \cdots
        u_{k-2}z_{k-2}u_{k-1}p_2q_2u_1
\]
forms a cycle of length $2k$. Indeed, the two pairs
$p_1q_1$ and $p_2q_2$ are edges of $H[X]$, while every other
edge in the displayed cycle joins some $u_i$ to a vertex of $X$.
Such an edge exists by the definition of $X$.
Now, we obtain a copy of $C_{2k}$ all of whose
vertices lie in $N(u_0)$, since
$u_1,\ldots,u_{k-1}\in N_A(u_0)$ and $X\subseteq N_B(u_0)$.
Consequently, $H$ contains a copy of $W_{2k+1}$,
which contradicts the assumption that $G$ is $W_{2k+1}$-free.
Thus, Case 2 is impossible.

We now conclude that every
$W_{2k+1}$-free graph $G$ satisfying
$\rho(G)\ge g_k(m)$ must fall into Case 1.
Hence, $\rho(G)=g_k(m)$
and $G\cong K_k\vee qK_1$, where
$m=\binom{k}{2}+kq$.
Therefore, every $W_{2k+1}$-free graph $G$ of sufficiently large size $m$ satisfies
$\rho(G)\le g_k(m)$.
Conversely, Proposition \ref{prop:w2k+1} shows that
$K_k\vee qK_1$ is $W_{2k+1}$-free and satisfies
$\rho(K_k\vee qK_1)=g_k(m)$
whenever $m=\binom{k}{2}+kq.$
This completes the proof of Theorem \ref{thm2}.
\end{proof}

\section{Concluding remarks}

Let $G$ be an $m$-edge $W_{2k+1}$-free graph
satisfying $\rho(G)\geq\vartheta_k(m)$.
By Lemma~\ref{lem:core-reduction}, $G$ contains
a connected $\varepsilon_0$-core $H$ such that
$e(H)=(1-o(1))m$ and
$\rho(H)\geq\vartheta_k(e(H))$.
By Lemma~\ref{thm:stability}, there exist disjoint
sets $A,B\subseteq V(H)$ such that
$d(H,K_{A,B})=o(e(H))$.
The estimates for this partition, together with
the $\varepsilon_0$-core inequalities, provide
bounds on crossing degrees and Perron coordinates.
These bounds are used to determine the equality
cases for $W_5$ and $W_{2k+1}$ with $k\geq3$.
As immediate consequences,
Theorem~\ref{thm2} also yields the known sharp edge-spectral bounds for
$C_{2k+1}$~\cite{Lixin} and for $K_1\vee kK_2$ and
$V_{2k+1}$~\cite{Lishuchao}.  This further suggests that the same
strategy may be applicable to other three-chromatic forbidden graphs
with suitable local embedding constraints.

The separation between global reduction and local analysis also
suggests a possible route to edge-spectral supersaturation.
Such an extension would require core and stability reductions that
remain effective when a small number of forbidden copies is allowed,
together with quantitative local embedding estimates that convert
spectral excess into lower bounds on the number of forbidden
subgraphs.
Some progress in this direction is already known.
For triangles, Ning and Zhai~\cite{NingZhai} obtained a sharp
edge-spectral supersaturation bound at the Nosal threshold, while
Chen, Li and Tang~\cite{ChenLiTang} established quantitative bounds
in terms of the spectral surplus $\rho(G)-\sqrt{m}$.
For odd cycles, Li, Zhai and Shu~\cite{Lixin} established the sharp
edge-spectral Tur\'an bound.
Subsequently, Fang, Lin and Zhai~\cite{FangLinZhai} determined the
correct order of magnitude of supersaturation for odd cycles and
$K_{s,t}^{+}$, and Liu, Lou and Sun~\cite{LiuLouSun} obtained the
asymptotically sharp leading constant for odd cycles.
These developments naturally lead to the following general problem.

For a fixed graph $F$, define
$\operatorname{ex}_{\rho}(m,F)
:=\max\{\rho(H):e(H)=m,\ H\text{ is }F\text{-free}\},$
and let $N_F(G)$ denote the number of copies of $F$ in $G$.

\begin{prob}\label{prob:chi3-supersaturation}
Let $F$ be a fixed graph with $\chi(F)=3$.
Determine the sharp lower bound on $N_F(G)$ for $m$-edge graphs $G$
satisfying
$\rho(G)>\operatorname{ex}_{\rho}(m,F)$,
and characterize the extremal graphs attaining this bound.
More generally, determine how $N_F(G)$ depends on the spectral excess
$\rho(G)-\operatorname{ex}_{\rho}(m,F)$.
\end{prob}

\end{document}